\documentclass[a4paper]{amsart}
\usepackage{graphicx}
\usepackage{bbm, caption, color, enumitem, fourier}
\usepackage{pdflscape}
\usepackage{tikz}
\usetikzlibrary{calc, intersections, positioning}
\def\ci{\protect{\perp\!\!\!\!\!\perp}}
\def\fdd{\stackrel{\textnormal{fdd}}{\Rightarrow}}

\def\d{\mathrm{d}}
\def\zero{{\bf{z}}}
\def\dualoperator{\vartheta}
\def\shift{\theta}

\def\C{{\mathbb{C}}}

\def\E{{\mathbb{E}}}

\def\H{{\mathbb{H}}}

\def\L{{\mathbb{L}}}

\def\N{{\mathbb{N}}}
\def\P{{\mathbb{P}}}
\def\Q{{\mathbb{Q}}}
\def\R{{\mathbb{R}}}

\def\Z{{\mathbb{Z}}}

\def\cF{{\mathcal F}}
\def\cG{{\mathcal G}}
\def\cH{{\mathcal H}}
\def\cL{{\mathcal L}}
\def\cM{{\mathcal M}}

\def\cP{{\mathcal P}}

\def\cz{\hat {\mathcal Z}}

\def\cT{{\hat T}}

\def\cV{{\mathcal V}}

\def\cZ{{\mathcal Z}}
\def\X{{\mathcal X}}

\def\eps{\varepsilon}

\def\bTm{{\widetilde T}^{-1}}

\newcommand{\Indicator}[1]{\mathbbm{1}\left(#1\right)}

\newcommand{\stub}[1]{
	\draw (-.1,-.1+#1) -- (.1,.1+#1);
	\draw (.1,-.1+#1) -- (-.1,.1+#1);
}

\newcommand{\laststub}[1]{
	\stub{#1}
	\draw[line width=2pt] (0,0) -- (0,#1);
}

\theoremstyle{plain}
\newtheorem{lemma}{Lemma}[section]
\newtheorem{theorem}[lemma]{Theorem}
\newtheorem{prop}[lemma]{Proposition}
\newtheorem{corollary}[lemma]{Corollary}

\newtheorem{innerassumption}{Assumption}
\newenvironment{assumption}[1]
{\renewcommand\theinnerassumption{#1}\innerassumption}
{\endinnerassumption}

\newtheorem{innertheorembis}{Theorem}
\newenvironment{theorembis}[1]
{\renewcommand\theinnertheorembis{#1}\innertheorembis}
{\endinnertheorembis}

\theoremstyle{remark}
\newtheorem{remark}[lemma]{Remark}

\title[Height and contour processes of Crump-Mode-Jagers forests (II)]{Height and contour processes of Crump-Mode-Jagers forests (II): The Bellman--Harris universality class}

\author{Emmanuel Schertzer}
\address{LPMA/UMR 7599\\Universit\'e Pierre et Marie Curie (P6) -- Bo\^ite courrier 188\\75252 PARIS Cedex 05 (FRANCE)}
\email{emmanuel.schertzer@upmc.fr}

\author{Florian Simatos}
\address{ISAE SUPAERO and Universit\'e de Toulouse\\10 avenue Edouard Belin\\31055 Toulouse Cedex 4\\France}
\email{florian.simatos@isae.fr}

\date{\today}
\numberwithin{equation}{section}
\begin{document}

\begin{abstract}
	Crump--Mode--Jagers (CMJ) trees generalize Galton--Watson trees by allowing individuals to live for an arbitrary duration and give birth at arbitrary times during their life-time. In this paper, we exhibit a simple condition under which the height and contour processes of CMJ forests belong to the universality class of Bellman--Harris processes. This condition formalizes an asymptotic independence between the chronological and genealogical structures. We show that it is satisfied by a large class of CMJ processes and in particular, quite surprisingly, by CMJ processes with a finite variance offspring distribution. Along the way, we prove a general tightness result.
\end{abstract}

\maketitle

\setcounter{tocdepth}{1}
\tableofcontents
\listoffigures

\newpage

\section{Introduction and presentation of results in the non-triangular case}

\subsection{Crump-Mode-Jagers forests} The subject of the present paper is the study of the height and contour processes of planar Crump--Mode--Jagers (CMJ) forests, which are random instances of \emph{chronological forests}. Chronological trees generalize discrete trees in the following way. Each individual $u$ is endowed with a pair $(V_u, \cP_u)$ such that:
\begin{enumerate}
	\item[(1)] $\cP_u$ is a point measure on $(0,\infty)$ where atoms represent the age of $u$ at childbearing, so that the mass $\lvert \cP_u \rvert$ of $\cP_u$ is the total number of children of $u$;
	\item[(2)] $V_u \in (0,\infty)$ represents the life-length of $u$ and satisfies $\cP_u(V_u, \infty) = 0$, i.e., individuals produce their offspring during their life-time.
\end{enumerate}

As noted by Lambert~\cite{Lambert10:0}, a chronological tree can be regarded as a tree satisfying the rule ``edges always grow to the right''. This is illustrated in Figures~\ref{fig:sequential-construction-sticks} and~\ref{fig:sequential-construction} where we present the sequential construction of a planar chronological forest from a sequence of ``sticks'' $\omega=(\omega_n, n\geq0)$, where $\omega_n = (V_n, \cP_n)$. For $n \geq 1$, $(V_n, \cP_n)$ describes the life of the $n$th individual according to the lexicographic order.

A CMJ forest is obtained when the initial sequence of sticks $((V_n, \cP_n), n \geq 0)$ is i.i.d.. In this introduction we aim to present our main results in a concise way. In particular, we explain our main results (corresponding to Theorems~\ref{thm:tightness}--\ref{thm:extended-BH} below) in the non-triangular setting where the common law of the $(V_n, \cP_n)$, denoted by $(V^*, \cP^*)$, is independent of the scaling parameter $p$ and satisfies $\E(\lvert \cP^* \rvert) = 1$. Except for Theorem~\ref{thm:finite-variance}, these results are extended to the triangular case in Theorems~\ref{thm:tightness'},~\ref{thm:height'},~\ref{thm:contour'} and~\ref{thm:extended-BH'} below. Moreover, our results involve some complex objects (such as the spine process) which are defined next only informally: formal and more cumbersome definitions are provided in Section~\ref{sec:notation}.

\subsection{Height, spine and contour of a CMJ forest}

The chronological height, spine and contour processes introduced now are illustrated on Figures~\ref{fig:y-spine} and~\ref{fig:chronological-processes}.

We define the chronological height process at time $n$, denoted $\H(n)$, as the date of birth of the $n$th individual in the forest. The chronological height can be obtained by summing up the ``chronological contribution'' of each ancestor along the ancestral line associated with $n$. To formalize this statement, we consider the spine at time $n$, denoted $\Pi(n)$, which is the measure recording each of those contributions along the spine. In particular, we have the relation $\H(n) = \lvert \Pi(n) \rvert$ with $\lvert \nu \rvert$ the mass of a measure $\nu$.

For the contour process, we follow Duquesne~\cite{Duquesne:0} (with the difference that we consider c\`adl\`ag instead of c\`agl\`ad coding functions) and consider an exploration particle traveling along the edges of the forest from left to right with the following convention: the particle travels at infinite speed when going downward and at unit speed when going upward (see Figure~\ref{fig:chronological-processes}). Usually, the contour process at time~$t$, denoted $\C(t)$, is defined as the distance of the particle to the root. Here we rather encode the chronological contour process by two coordinates and write $\C(t) = (\C^{\star}(t), \X(t))$ with $\C^\star(t)$ the chronological height of the individual visited at time $t$ by the exploration particle and $\X(t)$ the position of the exploration particle relatively to the edge currently visited. In particular, the distance of the particle to the root at time $t$ (i.e., the value of the classical contour process) is given by $\C^\star(t) + \X(t)$. The motivation for this decomposition is that we then have the explicit expression
\begin{equation} \label{eq:expression-C}
	\C^\star(t) = \H \circ \cV^{-1}(t) \ \text{ and } \ \X(t) = t - \cV(\cV^{-1}(t)-), \ t \geq 0,
\end{equation}
where $\cV$ is the renewal process $\cV(n) = V_0 + \cdots + V_n$ and $\cV^{-1}$ is its right-continuous inverse, see Figure~\ref{fig:chronological-processes}.

Moreover, we will denote by $\cH$ the genealogical height process defined similarly as above but from the genealogical forest constructed out of the sequence of sticks $(1, \lvert \cP_i \rvert \epsilon_1)$ where sticks are scaled to unit size and all atoms are gathered at the end (here and in the sequel, $\epsilon_x$ denotes the Dirac measure at $x$).

To summarize, we will consider the following processes:
\begin{description}
	\item[Genealogical process] $\cH$ denotes the \textit{genealogical} height process which encodes the forest constructed out of the sequence $(1, \lvert \cP_i \rvert \epsilon_1)$ -- see Figure~\ref{fig:genealogical-processes};
	\item[Chronological processes] $\H$ and $\C = (\C^\star, \C)$ denote the \textit{chronological} height and contour processes, respectively, and $\Pi =(\Pi(n), n\geq0)$ denotes the spine process.
\end{description}

\subsection{General tightness result}
\label{choice:scaling}

Our first main result is that, under the standard assumptions~\ref{ass-G},~\ref{ass-C1} and~\ref{ass-C2} below, there exists a ``good'' scaling of the chronological height and contour processes, i.e., one that leads to tight sequences with non-degenerate accumulation points. These assumptions are related to the tail of random variables $T(1)$ and $R(1)$ where, for a given ``typical'' individual $x$:
\begin{itemize}
	\item $T(1)$ is the lexicographic distance to her first ancestor;
	\item $R(1)$ is the age of her parent when begetting $x$.
\end{itemize}
Formally, those random variables are defined through the ladder height process associated to the Lukasiewicz path 
\[ S(n) = \sum_{i=0}^{n-1} (|\cP_i| -1) \]
and their joint law is described in more details in~\eqref{eq:F} below. Assumption~\ref{ass-G} concerns the genealogical structure, and Assumptions~\ref{ass-C1} and~\ref{ass-C2} the chronological structure.

\begin{assumption} {\text{\textnormal{G}}} \label{ass-G}
	We say that Assumption~\ref{ass-G} holds if $\E( \lvert \cP^* \rvert^2)<\infty$ or if $\lvert \cP^* \rvert$ is in the domain of attraction of a $\gamma$-stable distribution with $\gamma\in(1,2)$.
\end{assumption}

\begin{assumption} {\text{\textnormal{C1}}} \label{ass-C1}
	We say that Assumption~\ref{ass-C1} holds if $R(1)$ is in the domain of attraction of a $\beta$-stable distribution with $\beta \in (0,1)$.
\end{assumption}

\begin{assumption} {\text{\textnormal{C2}}} \label{ass-C2}
	We say that Assumption~\ref{ass-C2} holds if $V^*$ is in the domain of attraction of an $\alpha$-stable law with $\alpha\in(0,1)$.
\end{assumption}

Actually, our results also apply when $R(1)$ and $V^*$ have finite means. However, this case has already been treated in~\cite{Schertzer:0} and leads to somewhat degenerate limits within the framework of the present paper (deterministic subordinators), see Section~\ref{sect:discussion} for more details.

Let $X,Y,Z$ be stable laws with respective Laplace exponents $\lambda^\gamma, \lambda^\beta, \lambda^\alpha$. Under Assumptions~\ref{ass-G},~\ref{ass-C1} and~\ref{ass-C2}, there exist vanishing scaling sequences $(g_p)$, $(c_p)$ and $(v_p)$ such that 
\[g_p \sum_{i=1}^p (|\cP_i| -1) \Rightarrow X, \ c_p \sum_{i=1}^{[1/g_p]} R(i) \Rightarrow Y \ \text{ and } \ v_p\sum_{i=1}^{p} |V_i| \Rightarrow Z. \]
More precisely, these scaling sequences can be written
\[g_p = p^{-1/\gamma} \ell(p), \ c_p = g_p^{1/\beta} \ell'(p) \ \mbox{and} \ v_p \ = \ p^{-1/\alpha} \ell''(p)\] 
with $\gamma \in (1,2]$ ($\gamma = 2$ corresponding to the finite variance case) and $\ell,\ell'$ and $\ell''$ slowly varying functions. Armed with these scaling sequences we then scale the renewal process $\cV$ and the height and contour processes as follows:
\[ \cV_p(t) = v_p \cV([pt]), \ \ \H_p(t) = c_p \H([pt]), \ \ \C^\star_p(t) = c_p \C^{\star}(pt) \ \ \text{ and } \ \ \X_p(t) = v_p \X(pt) \]
for $t \geq 0$. Note in particular that the two coordinates of the contour do not have the same scaling, and we will actually show that, under general assumptions, we have $c_p \gg v_p$, which informally corresponds to the fact that the edge currently visited is much longer than the height of its bottom point.

Finally $\Pi_p$ will denote the spine process when time is scaled by $p$, space by $g_p$ and mass by $c_p$: $\Pi_p(t)([0,x]) = c_p \Pi([pt])([0,x/g_p])$ (recall that $\Pi(n)$ is a measure). The process $\Pi_p$ will be seen as a path-valued process, see Section~\ref{sub:snake} for more details and in particular for the topology considered.

\begin{theorembis} {A} \label{thm:tightness}
	Under Assumptions~\ref{ass-G},~\ref{ass-C1} and~\ref{ass-C2} and the above scaling, the sequence $(\Pi_p, \cV_p)$ is tight. Moreover, let $(\Pi_\infty, \cV_\infty)$ be any accumulation point and for $t \geq 0$ define $\H_\infty(t) = \lvert \Pi_\infty(t) \rvert$, $\C^\star_\infty(t) = \H_\infty \circ \cV_\infty^{-1}(t)$ and $\X_\infty(t) = t - \cV_\infty(\cV_\infty^{-1}(t)-)$. Then:
	\begin{enumerate}
		\item $\Pi_\infty$ is almost surely continuous, and $\cV_\infty$ is a subordinator with Laplace exponent~$\lambda^\alpha$;
		\item for every $t \geq 0$, $\H_\infty(t)$ and $\C^\star_\infty(t)$ are almost surely finite and strictly positive;
		\item if $(\Pi_p, \cV_p) \Rightarrow (\Pi_\infty, \cV_\infty)$ along a subsequence, then for any finite set $I \subset \R_+$ we have
		\begin{equation} \label{eq:full-conv}
			\left(\Pi_p, \cV_p, (\H_p(t))_{t \in I}, (\C^\star_p(t))_{t \in I}, \X_p \right) \Rightarrow \left(\Pi_\infty, \cV_\infty, (\H_\infty(t))_{t \in I}, (\C^\star_\infty(t))_{t \in I}, \X_\infty \right)
		\end{equation}
		along the same subsequence.
	\end{enumerate}
\end{theorembis}

In the Galton--Watson case, it is well-known that the height and contour processes are related by a deterministic time change~\cite{Duquesne02:0} (as is suggested by the above result and was shown in~\cite{Schertzer:0}, this continues to hold as soon as $V^*$ has finite mean). The convergence~\eqref{eq:full-conv} thus states that a similar result holds for general CMJ's, where the first coordinate of the contour is obtained by a time change of $\H_\infty$ expressed in terms of a subordinator $\cV_\infty$.

\subsection{The Bellman--Harris case}\label{BH-description}

In the Bellman--Harris case, an individual gives birth at her death and the number of children is independent from the life length, i.e., $\cP_u = \xi_u \epsilon_{V_u}$ with $\epsilon_x$ the Dirac measure at $x$ and $\xi_u$ an integer-valued random variable independent from $V_u$. Bellman--Harris branching processes have received considerable attention in the literature, in part owing to the fact that they are the simplest tractable non-Markovian branching processes.

The chronological tree corresponding to a Bellman--Harris process can be obtained by putting i.i.d.\ marks (distributed as $V^*$) on the edges of the corresponding genealogical tree, that can be seen as stretching factors. From this viewpoint, these trees are particular cases of branching random walks which are obtained in a similar manner but without the positivity constraint on the marks.

Upon normalization, it is known since Aldous~\cite{Aldous93:0}, see also~\cite{Duquesne02:0, Gittenberger98:0, Marckert03:0}, that the genealogical tree converges to a random tree called L\'evy tree (the famous Brownian continuous random tree in the case of offspring distribution with finite variance).

Given the above relation between the chronological and genealogical trees in the Bellman--Harris case, in this case it is natural to expect the CMJ tree to converge toward a random tree obtained by marking the limiting genealogical L\'evy tree and then using these marks to stretch the corresponding portion of the tree. This intuition has been carried out for branching random walks with finite variance offspring distribution by Marckert and Mokkadem~\cite{Marckert03:1}, see also~\cite{Gittenberger03:0, Janson05:0}. Formally, the height process of the limiting tree is then described as the terminal value of a Brownian snake.

\subsection{Main result: the Bellman--Harris universality class}

Our second main result is a generalization of the above picture for a much wider class of CMJ forests. In the Bellman--Harris case, the Brownian snake arises because of the independence between $V^*$ and $\lvert \cP^* \rvert$, i.e., between the chronological and genealogical structures. In this paper we identify general asymptotic independence conditions (\eqref{eq:independence-condition} and~\eqref{eq:independence-condition-2} below) under which the limiting height process is, as in the Bellman--Harris case, the terminal value of a snake.

\subsubsection{The $\lambda^\gamma/\lambda^\beta$ snake} As in~\cite{Gittenberger03:0, Janson05:0, Marckert03:1} this snake is a Brownian snake in the case of finite variance offspring distribution but in the case $\gamma \in (1,2)$, it will be a snake whose lifelength process is the height process associated to the L\'evy process with Laplace exponent $\lambda^\gamma$ and whose spatial displacement is given by a subordinator with Laplace exponent $\lambda^\beta$: this snake will be called the $\lambda^\gamma/\lambda^\beta$-snake, see Section~\ref{sub:snake} and~\cite{Duquesne02:0} for more details.

Informally, this object is the continuum counterpart of the description of the Bellman--Harris forests spelled out in the last paragraph of Section \ref{BH-description}. It can be generated by considering a $\gamma$-stable L\'evy tree, which encodes the genealogy, and by marking the tree according to a Poisson point process with intensity measure $\d\lambda \times \Indicator{x > 0} x^{-\beta} \d x$ where $\lambda$ denotes the branch-length measure on the real tree~\cite{Evans08:0}. In particular, the first coordinate of each point of the Poisson point process is a location in the tree, whereas the second coordinate is a positive real number interpreted as a mark (or stretching factor in the spirit of the previous section). The spine process $\Pi_\infty(t)$ associated to the snake is then obtained by considering the marks lying along the ancestral line associated to the point $t$. See Section~\ref{sub:snake} for a definition.

\subsubsection{The Bellman--Harris universality class} The conditions that generalize the independence condition in the Bellman--Harris case are:
\begin{equation} \label{eq:independence-condition}
	\P \left( c_p R(1) \geq \varepsilon \mid T(1) \geq \delta p \right) \to 0 \ \text{ for every } \ \varepsilon, \delta > 0 \tag{IC1}
\end{equation}
and
\begin{equation} \label{eq:independence-condition-2}
	\P \left( \lvert \cP^* \rvert \geq \varepsilon p g_p \mid v_p V^* \geq \delta \right) \to 0 \ \text{ for every } \ \varepsilon, \delta > 0. \tag{IC2}
\end{equation}
In the Bellman--Harris case the independence structure implies that
\[ \P \left( c_p R(1) \geq \varepsilon \mid T(1) \geq \delta p \right) = \P \left( c_p R(1) \geq \varepsilon \right) \]
and
\[ \P \left( \lvert \cP^* \rvert \geq \varepsilon p g_p \mid v_p V^* \geq \delta \right) = \P \left( \lvert \cP^* \rvert \geq \varepsilon p g_p \right) \]
so that~\eqref{eq:independence-condition} and~\eqref{eq:independence-condition-2} are indeed satisfied as $c_p \to 0$ while $p g_p \to \infty$.

\begin{theorembis} {B} \label{thm:height}
	If Assumptions~\ref{ass-G},~\ref{ass-C1} and~\ref{ass-C2} are satisfied and the asymptotic independence condition~\eqref{eq:independence-condition} holds, then $\Pi_p \Rightarrow \Pi_\infty$ where $\Pi_\infty$ is the $\lambda^\gamma/\lambda^\beta$ snake. In particular, $\H_p$ has the same scaling limit than the scaled height process of a well-chosen Bellman--Harris forest.
\end{theorembis}

Combining Theorems~\ref{thm:tightness} and~\ref{thm:height}, we see that when~\ref{ass-G},~\ref{ass-C1},~\ref{ass-C2} and~\eqref{eq:independence-condition} hold, then $(\Pi_p, \cV_p)$ is tight and the marginals of any accumulation point $(\Pi_\infty, \cV_\infty)$ are specified: $\Pi_\infty$ is the $\lambda^\gamma/\lambda^\beta$-snake and $\cV_\infty$ is a subordinator with Laplace exponent $\lambda^\alpha$. As $\H_\infty(t) = \lvert \Pi_\infty(t) \rvert$ this is enough to describe the law of $\H_\infty$ (which by Theorem~\ref{thm:tightness} is the scaling limit of $\H_p$) but since $\C^\star_\infty = \H_\infty \circ \cV^{-1}_\infty$ one needs to specify the correlation structure between $\Pi_\infty$ and $\cV_\infty$ in order to determine the scaling limit of $\C^\star_p$. Our next main result shows that if in addition~\eqref{eq:independence-condition-2} holds, then $\H_\infty$ and $\cV_\infty$ are independent, thereby implying the finite-dimensional convergence of $(\H_p, \C^\star_p)$ (and thus of $(\H_p, \C^\star_p, \X_p)$).

\begin{theorembis} {C} \label{thm:contour}
	If Assumptions~\ref{ass-G},~\ref{ass-C1} and~\ref{ass-C2} are satisfied and the asymptotic independence conditions~\eqref{eq:independence-condition} and~\eqref{eq:independence-condition-2} hold, then $\cV_\infty$ is independent of $\Pi_\infty$. Moreover, in this case we have $c_p / v_p \to \infty$.
\end{theorembis}

Under the assumptions of Theorem~\ref{thm:contour}, we therefore have $(\C^\star_p, \X_p) \fdd (\C^\star_\infty, \X_\infty)$. This convergence suggests that asymptotically, the chronological contour process $\C$ can be seen as a nice, baseline process $\H_\infty \circ \cV_\infty^{-1}$ on the space scale $1/c_p$, to which long ``hairs'' of the order $v_p \gg 1/c_p$ are grafted. Such a behavior was first established for Bellman--Harris processes in~\cite{Vatutin79:1} and then in a more precise form for branching random walks~\cite[Theorem $5$]{Janson05:0}.

\subsection{Explicit examples} \label{sub:examples}
We now describe two large classes of models where the above assumptions are satisfied. The first class consists of CMJ forests with finite variance offspring distribution. We find it quite striking that such a simple condition, without any assumption on how atoms of $\cP^*$ are spread, implies that the corresponding CMJ forests belong to the Bellman--Harris universality class.

\begin{theorembis} {D} \label{thm:finite-variance}
	If Assumptions~\ref{ass-G},~\ref{ass-C1} and~\ref{ass-C2} are satisfied with $\E(\lvert \cP^* \rvert^2)<\infty$, then~\eqref{eq:independence-condition} and~\eqref{eq:independence-condition-2} hold. In particular, the limiting spine process is a $\lambda^2/\lambda^\beta$ snake and the time-change $\cV_\infty$ is an independent subordinator with Laplace exponent $\lambda^\alpha$.
\end{theorembis}

The second class is a natural extension of the Bellman--Harris case, and allows $\lvert \cP^* \rvert$ to have infinite variance.

\begin{theorembis} {E} \label{thm:extended-BH}
	Assume that:
	\begin{itemize}
		\item Assumptions~\ref{ass-G} and~\ref{ass-C2} are satisfied with $V^*$ and $\lvert \cP^* \rvert$ independent;
		\item conditionally on $(V^*, \lvert \cP^* \rvert) = (v,n)$, the locations of the atoms of $\cP^*$ are i.i.d.\ with common distribution $v X$ for some random variable $X \in (0,1]$.
	\end{itemize}
	Then Assumption~\ref{ass-C1} is satisfied with $\beta = \alpha$ and~\eqref{eq:independence-condition} and~\eqref{eq:independence-condition-2} hold. In particular, $\Pi_\infty$ is the $\lambda^\gamma/\lambda^\alpha$-snake and the time-change $\cV_\infty$ is an independent subordinator with Laplace exponent $\lambda^\alpha$.
\end{theorembis}

Note that this result generalizes results for Bellman--Harris processes which correspond to the case $X = 1$.

\subsection{Beyond the Galton--Watson and Bellman--Harris universality classes}\label{sect:discussion}

As presented above, our main results assume that $R(1)$ and $V^*$ have infinite first moment. When both $R(1)$ and $V^*$ have finite first moment, all the edges are ``short'' and the CMJ forest is close to its genealogical counterpart: in other words, the time structure does not matter much. This case of ``short edges'' has been worked out in~\cite{Schertzer:0}. The main result is that CMJ forests with $\E(R(1)) < \infty$ and $\E(V^*) < \infty$ belong to the universality class of Galton--Watson processes and that in the limit, the chronological height and contour processes are related through a deterministic time-change.

Beyond the Galton--Watson and Bellman--Harris universality classes treated in~\cite{Schertzer:0} and in the present paper, there remains a large class of CMJ forests with ``long'' edges but where the chronological and genealogical structures remain dependent in the limit (i.e.,~\eqref{eq:independence-condition} or~\eqref{eq:independence-condition-2} does not hold). In current work in progress, we are looking at the case where $\cP^*$ conditionally on $V^*$ is a renewal process stopped at $V^*$: the chronology and the genealogy then remain positively correlated. Whether or not even more general results can be obtained in this dependent setting is an interesting research direction.

\subsection{Comments on the contour process}

There are various definitions of the contour process in the literature. In every case, the value of the contour is defined as the distance to the root of an exploration particle, but various choices have been made for the speed at which this particle explores the tree.

For discrete trees and in~\cite{Schertzer:0}, the particle is assumed to travel at unit speed along the edges of the forest, in such a way that each point of the tree is visited twice.

Alternatively, Lambert proposed in~\cite{Lambert10:0} the Jumping Chronological Contour Process where the particle travels at infinite speed when going upward and at speed one when going downward. In the binary, homogeneous case where $\cP^*$ is an independent Poisson process stopped at $V^*$, this contour process has the desirable property of being a L\'evy process. This property has far-reaching implications which allow for a detailed study of this important class of CMJ processes, see for instance~\cite{Cloez:0, Davila-Felipe15:0, Delaporte15:0, Lambert15:0, Lambert13:0, Richard13:0, Richard14:0}.

In the present paper, analogously to~\cite{Duquesne:0} we make the exploration particle move at infinite speed when going downward, and at unit speed when going upward, see Figure~\ref{fig:chronological-processes}. We believe that this is a good choice as far as scaling limits are concerned because of the relations~\eqref{eq:expression-C}. In~\cite{Schertzer:0}, we defined the contour process in the ``classical way'', i.e., when the exploration particle moves at constant speed. This definition led to a significant number of technical problems due to the absence of tightness of the contour in the Skorohod topology. Our choice of contour function allows to circumvent these problems and we believe that it has potential for broader applications.

\section{Notation and set-up} \label{sec:notation}

\subsection{Notation} \label{sub:notation} We gather here the notation used in the rest of the paper. Let $\Z$ denote the set of integers, $\N = \Z \cap \R_+$ the set of non-negative integers and $\Q$ the set of rational numbers. For $x, y \in \R$ let $[x] = \max\{n \in \Z: n \leq x\}$, $x^+ = \max(x, 0)$ and $x \wedge y = \min(x, y)$. Throughout we adopt the convention $\max \emptyset = \sup \emptyset = -\infty$ and $\min \emptyset = \inf \emptyset = +\infty$.

\subsubsection{Functions}

The set of c\`adl\`ag functions is endowed with the Skorohod topology. For a c\`adl\`ag function $f: \R \to \R$ we denote by $f(x-)$ its left-limit at $x \in \R$, by $\Delta f(x) = f(x) - f(x-)$ the size of its jump and by $f^{-1}$ its right-continuous inverse:
\[ f^{-1}(t) = \inf \left\{s \geq 0: f(s) > t \right\}, \ t \geq 0. \]
For any $h \geq 0$ and any two mappings $f, g$ defined on $\R_+$ or a subset thereof (such that $f$ is at least defined on $[0,h]$) we define $[f, g]_h$ by
\begin{equation} \label{eq:conc}
	[f, g]_h(x) =
	\begin{cases}
		f(x) & \text{ if } x \leq h,\\
		f(h) + g(x-h) & \text{ if } x > h.
	\end{cases}
\end{equation}
For $h \in \R$ we consider $\Theta_h$, $\cdot \mid_h$ and $\cdot \mid_{h-}$ the shift and stopping operators, which act on functions $f : \R \to \R$ by
\[ \Theta_h(f)(x) =
	\begin{cases}
		f(x+h) - f(h) & \text{ if } x > 0,\\
		0 & \text{ if } x \leq 0,
	\end{cases} \]
and
\[ f\mid_h(x) = f \left( x \wedge h \right) \ \text{ and } \ f\mid_{h-}(x) = f \mid_h(x) - \Delta f(h) \Indicator{x \geq h} = \begin{cases}
	f(x) & \text{ if } x < h,\\
	f(h-) & \text{ if } x \geq h.
\end{cases} \]
Note in particular that $\Theta_h(f)(0) = 0$, $f\mid_h(h) = f(h)$ and $f\mid_{h-}(h) = f(h-)$.

\subsubsection{Measures}

We let $\cM$ be the set of positive Radon measures on $[0,\infty)$ endowed with the weak topology, $\epsilon_x \in \cM$ for $x \geq 0$ be the Dirac measure at $x$ and $\zero$ be the zero measure. We will identify any measure $\nu \in \cM$ with the c\`adl\`ag, non-decreasing function $x \in \R \mapsto \nu[0,x^+] \in \R_+$ which gives sense to $\Theta_h(\nu)$, $\nu\mid_h$, $\nu\mid_{h-}$ and $[\nu, \nu']_h$ for $\nu, \nu' \in \cM$ and $h \in \R$, e.g.,
\[ \Theta_h(\nu)(x) = \nu \left( (h, x+h] \cap \R_+ \right), \ \nu\mid_h(x) = \nu \big( [0,h] \cap [0,x] \big) \ \text{ and } \ \left[\nu,\nu'\right]_h \ = \ \nu \mid_h + \Theta_{-h}(\nu') \]
for $x \geq 0$. We will also write
\[ \left[\nu,\nu'\right]_{h-} \ = \ \nu \mid_{h-} + \Theta_{-h}(\nu') \]
so that $[\nu, \nu']_h = [\nu, \nu']_{h-} + \nu(\{h\}) \epsilon_h$. Note also that we have $\nu = [\nu, \Theta_h(\nu)]_h$.

The mass of $\nu \in \cM$ will be denoted by $\lvert \nu \rvert = \nu[0,\infty)$ and the supremum of its support by $\pi(\nu) = \inf \{ x \geq 0: \pi(x, \infty) = 0 \}$ with the convention $\pi(\zero) = 0$. When $\lvert \nu \rvert$ and $\pi(\nu)$ are finite we consider the reversed measure $\cL(\nu)$ defined by
\[ \cL(\nu)(x) = \nu[\pi(\nu) - x, \pi(\nu)] = \lvert \nu \rvert - \nu[0, \pi(\nu) - x), \ x \geq 0. \]

If $\nu \in \cM$ is of the form $\nu = \sum_{i=1}^{\lvert \nu \rvert} \epsilon_{a(i)}$ with $0 \leq a(1) \leq \cdots \leq a(\lvert \nu \rvert))$ and if $k\in\{0, \ldots, \lvert \nu \rvert-1\}$, we will write $A_k(\nu)=a({k+1})$ for the position of the $(k+1)$st atom of $\nu$ where atoms are ranked from bottom to top.

In the following we need continuity properties of some of the above operators. The following lemma gathers the required results, which can be proved using standard results of the Skorohod topology, see for instance~\cite{Jacod03:0}.

\begin{lemma}\label{lemma:continuity-properties}
	If $f_p \to f$, $t_p \to t$ and either $\Delta f_p(t_p) \to \Delta f(t)$ or $\Delta f(t) = 0$, then $f_p \mid_{t_p} \to f\mid_t$, $f_p \mid_{t_p-} \to f\mid_{t-}$ and $\Theta_{t_p}(f_p) \to \Theta_t(f)$.
	
	If $\nu_p \to \nu$ with $\pi(\nu) = \infty$, $h_p \to h$ and $\nu$ has no atom at $h$, then $\cL(\nu_p \mid_{h_p}) \to \cL(\nu \mid_h)$ as well as $\cL(\nu_p \mid_{h_p-}) \to \cL(\nu \mid_h)$.	
\end{lemma}

\subsubsection{Random variables.}

Let
\[ \L = \{ (v, \nu) \in (0,\infty) \times \cM: \nu\{0\} = 0 \text{ and } v \geq \pi(\nu) \}. \]
We will start from discrete processes defined on the measurable space $(\Omega, \mathcal{F})$ with $\Omega = \L^\Z$ the space of doubly infinite sequences of sticks and $\mathcal{F}$ the $\sigma$-algebra generated by the coordinate mappings. An elementary event $\omega \in \Omega$ is written as $\omega = (\omega_n, n \in \Z)$ and $\omega_n = (V_n, \cP_n)$. For $n \in \Z$ we consider the operators $\shift_n, \dualoperator^n: \Omega \to \Omega$ defined as follows:
\begin{itemize}
	\item $\shift_n$ is the shift operator, defined by $\shift_n(\omega) = (\omega_{n + k}, k \in \Z)$;
	\item $\dualoperator^n$ is the dual (or time-reversal) operator, defined by $\dualoperator^n(\omega) = (\omega_{n - k - 1}, k \in \Z)$.
\end{itemize}
When no confusion can arise we will use the notation $W^n = W \circ \shift_n$ and $\hat W^n = W \circ \dualoperator^n$ for any random variable $W$ defined on $\Omega$.

We define $\sigma(X)$ as the $\sigma$-algebra generated by the random variable $X$, $\cF_{\geq \Gamma} = \sigma (\omega_k, k \geq \Gamma)$ and $\cF_{< \Gamma} = \sigma(\omega_k, k < \Gamma)$ for any random time $\Gamma: \Omega \to \N$, and $m \cF$ the set of random variables that are measurable with respect to the $\sigma$-algebra $\cF$. For random variables $X, Y$ we will write $X \ci Y$ to mean that they are independent.

\subsection{Lukasiewicz path and ladder process.}\label{sect:luka} In this section, we introduce an important trivariate renewal process $(\cZ, T, R)$: $(\cZ, T)$ is the usual ladder process of the Lukasiewicz path, and $R$ adds chronological information. We define the Lukasiewicz path $S = (S(n), n \in \Z)$ by $S(0) = 0$ and, for $n \geq 1$,
\[ S(n) = \sum_{k=0}^{n-1} ( \lvert \cP_k \rvert - 1 ) \ \text{ and } \ S(-n) = -\sum_{k=-n}^{-1} ( \lvert \cP_k \rvert - 1 ). \]
Define the ladder time process $T = (T(k), k \in \N)$ by $T(0)=0$ and for $k \geq 0$,
\[ T(k+1) = \inf \big\{ \ell > T(k): S(\ell) \geq S(T(k)) \big\} = T(1) \circ \shift_{T(k)} + T(k) \]
with the convention $T(k+1) = \infty$ if $T(k) = \infty$. We identify $T$ and the function $x \in \R_+ \mapsto T([x])$, and thus also with the measure $\sum_{k \in \N} (T(k+1) - T(k)) \epsilon_{k+1}$. We consider the following two inverses of $T$:
\[ T^{-1}(n) = \min \left\{ k \geq 0: T(k) \geq n \right\} \ \text{ and } \ \bTm = \max \left\{ k \geq 0: T(k) \leq n \right\}. \]
Note that $T^{-1}(n) = \bTm(n)$ if $n$ is a weak record time and $T^{-1}(n) = \bTm(n)+1$ otherwise. Moreover, it is well-known that in the Galton--Watson case, the height process at time $n$ is related to the renewal process $T$ through the relation $\cH(n) = \tilde T^{-1}(n)\circ\vartheta^n$.

Consider the ladder height process $\cZ$, defined (as a measure) by
\[ \cZ(\{k\}) = S(T(1)) \circ \theta_{T(k-1)} \]
for $k \geq 1$ with $T(k-1) < \infty$. In words, $\cZ(\{k\})$ is the value of the $k$th overshoot of $S$. We can now define the process $R$ that contains the useful chronological information (recall that $A_k(\nu)$ denotes the position of the $(k+1)$st atom of $\nu$):
\[ R = \sum_{k \in \N^{*}: T(k) < \infty} A_{\cZ(\{k\})}(\cP_{T(k)-1}) \ \epsilon_{k}. \]
As we shall see in Theorem~\ref{thm:H-Pi} below, $\hat R^n = R\circ \vartheta^n$ makes it possible to recover the chronological contribution of the successive ancestors of $n$ to its chronological height.

The strong Markov property implies that $(T, \cZ, R)$ is a trivariate renewal process. With the notation of the present paper, the equation below $(2.18)$ in~\cite{Schertzer:0} states that
\begin{equation} \label{eq:F-0}
	\E \big[ F \left(\cP_{T(1)-1}, - S(T(1)-1), T(1) \right) \big] = \sum_{t \geq 1, x \geq 0} \E \left( F(\cP^*, x, t) ; \lvert \cP^* \rvert \geq x+1 \right) \P \left( \tau^-_x = t \right)
\end{equation}
with $\tau^-_x = \inf\{n \geq 0: S(n) = -x\}$, from which we deduce that for any $G: \N \times \N \times \R_+ \to \R_+$ measurable, we have
\begin{equation} \label{eq:F}
	\E \big[ G(T(1), \cZ(1), R(1)) \big] = \sum_{t, x \geq 1} \sum_{z \geq 0} \E \left[ G(t, z, A_z(\cP^*)) ; \lvert \cP^* \rvert = x+z \right] \P \left( \tau^-_{x-1} = t-1 \right).
\end{equation}

\subsection{Spine process}\label{sect:spine-prcocess}

The spine process was introduced in~\cite{Schertzer:0} as an extension of the classical exploration process~\cite{Le-Gall98:0} suited for the chronological setting. Here, we only need a marginal of this process, denoted $\Pi = (\Pi(n), n \geq 0)$, which for simplicity we continue to name spine process. For each $n \geq 0$, $\Pi(n) \in \cM$ is the random measure defined by
\begin{equation} \label{eq:spine-length}
	\Pi(n) = \sum_{i=1}^{\cH(n)} \hat R^{n}(\{i\}) \ \epsilon_{\cH(n)-i} = \left( \sum_{i=1}^{\bTm(n)} R(\{i\}) \ \epsilon_{\bTm(n)-i} \right) \circ \dualoperator^n.
\end{equation}

\subsection{Joint distribution}

One of the main result of~\cite{Schertzer:0} is to relate the height process $\H$ as well as the spine process $\Pi$ to the bivariate renewal process $(T, R)$. Namely, with the notation of the present paper, Theorem~$1.1$ and Proposition~$2.4$ of~\cite{Schertzer:0} can be formulated as follows.

\begin{theorem} \label{thm:H-Pi}
	For any $n \geq 0$ we have
	\begin{equation} \label{eq:H-Pi}
		\H(n) = \left( R \circ \bTm(n) \right)\circ\vartheta^n, \ \cH(n) = \bTm(n) \circ\vartheta^n \ \text{ and } \ \Pi(n) = \cL\left( R \mid_{\bTm(n)} \right) \circ \vartheta^n.
	\end{equation}
	In particular, $\H(n)$ can be recovered by summing up the weights carried by the atoms of the measure $\Pi(n)$, i.e.,
	\[ \H(n) = \lvert \Pi(n) \rvert = \Pi(n)(\cH(n)-) = \sum_{i=1}^{\cH(n)} \ \hat R^{n}(\{i\}) = \hat R^n \left( \cH(n) \right). \]
\end{theorem}

\subsection{The $\psi/\phi$-snake} \label{sub:snake}

Let $\psi$ with $\int^\infty \d u / \psi(u) < \infty$ be the Laplace exponent of a spectrally positive L\'evy process with infinite variation that does not drift to $+\infty$. We call $\psi$-height process the height process associated to the L\'evy process with Laplace exponent $\psi$ and $\psi/\phi$-snake the L\'evy snake whose life-time process is the $\psi$-height process and whose spatial displacement is the L\'evy process with Laplace exponent $\phi$, see Duquesne and Le Gall~\cite{Duquesne02:0}.

The $\psi$-height process encodes the genealogy of the forest. As explained in the introduction, the scaling limit of the chronological forest will be obtained by marking the genealogy where the marks correspond to random stretching, which is exactly the interpretation of the $\psi/\phi$-snake when $\phi$ is the L\'evy exponent of a subordinator as will be the case here.
\\

Let a killed path be a c\`adl\`ag mapping $w: [0, \zeta) \to \R$ with $\zeta \in (0,\infty)$ called the life-time of the path and $\mathcal{W}$ be the set of killed paths, which is the state-space of the $\psi/\phi$-snake. 

Whenever it exists in $\R \cup \{\pm \infty\}$, we call terminal value of a killed path the value of $w(\zeta-)$. For two killed paths $w, w' \in \mathcal{W}$ with respective life times $\zeta$ and $\zeta'$, we consider the distance
\begin{equation}\label{eq:distance}
	d \left( w, w'\right) = \left \lvert \zeta - \zeta'\right \rvert + \int_0^{\zeta \wedge \zeta'} d_{(u)} (w_{\leq u}, w'_{\leq u}) \wedge 1 \ \d u,
\end{equation}
where $w_{\leq u}$ is the restriction of the path $w$ to $[0,u]$ and $d_{(u)}(f, g)$ denotes the Skorohod distance on the set of real-valued, c\`adl\`ag functions defined on $[0,u]$. As mentioned in~\cite{Abraham02:0}, $(\mathcal{W}, d)$ is a Polish space. We will use the following characterization of the $\psi/\phi$-snake.

\begin{theorem} \label{teo:exist-unique-snake}
	The $\psi/\phi$-snake is the unique $\mathcal{W}$-valued continuous process $\left(W(t), t\geq 0 \right)$ satisfying the two following properties:
	\begin{enumerate}[label=\textnormal{(PS-R-\Roman*)}, ref=\textnormal{(PS-R-\Roman*)}, leftmargin=20pt]
		\item[(1)] the life-time process $H$ is the $\psi$-height process;
		\item[(2)] conditionally on $H$, $\left(W(t), \ t\geq0\right)$ is the $\mathcal{W}$-valued time-inhomogeneous Markov process with the following transition mechanism. For every $0 \leq s \leq t$,
		\begin{equation} \label{eq:continuous-snake-forward}
			W(t) = \left[W(s), \tilde W \right]_{\min_{[s,t]} H} \ \mbox{in law}
		\end{equation}
		where $\tilde W$ is a L\'evy process independent from $W(s)$ with Laplace exponent $\phi$ and killed at $H(t)-\min_{[s,t]} H$.
	\end{enumerate}
\end{theorem}

\subsection{Triangular setting}

Recall that the main results announced in the introduction are stated in a non-triangular setting. In the following, we consider a triangular setting where $\P_p$ is the law under which the common law of the $(V_k, \cP_k)$ is equal to $(V^*_p, \cP^*_p)$. When considering convergence, we are then implicitly working under $\P_p$, that is, $X_p \Rightarrow X_\infty$ means that $\E_p(f(X_p)) \to \E(f(X_\infty))$ for every bounded continuous function $f$. Likewise, $\fdd$ refers to convergence of the finite-dimensional distributions under $\P_p$.

\section{General tightness result} \label{sec:general}

\subsection{Scaling} 
\label{sub:scaling-1}

As in the introduction, we consider throughout three vanishing scaling sequences $(g_p)$, $(c_p)$ and $(v_p)$ used to scale the genealogical and chronological processes as follows:
\[ \cH_p(t) = g_p \cH([pt]) \ \text{ and } \ S_p(t) = \frac{1}{p g_p} S([pt]), \]
\[ \C_p = (\C^\star_p, \X_p) \ \text{ with } \ \C^\star_p(t) = c_p \C^\star(pt) \ \text{ and } \ \X_p(t) = v_p \X_p(pt), \]
and
\[ \H_p(t) = c_p \H([pt]), \ \Pi_p(t)(x) = c_p \Pi([pt])(x / g_p) \ \text{ and } \ \cV_p(t) = v_p\cV([pt]) \]
with as before $\cV(n) = V_0 + \cdots + V_n$ for $n \geq 0$. In the sequel, we will consider $\Pi_p$ as a $\mathcal{W}$-valued process since the definition~\eqref{eq:spine-length} of $\Pi(n)$ implies $\Pi_p(t)[\cH_p(t), \infty) = 0$ and thus allows to identify $\Pi_p(t)$ with the killed path $x \in [0, \cH_p(t)) \mapsto \Pi_p(t)[0,x]$. Scale in addition the processes $T, R$ and $\cZ$ as
\[ T_p(x) = \frac{1}{p} T(x/g_p),\ R_p(x) = c_p R(x/g_p) \ \mbox{and } \cZ_p(x) = \frac{1}{p g_p} \cZ(x/g_p) , \ x \geq 0, \]
and finally define the dual and shifted versions of $T_p$, $R_p$ and $\cZ_p$ as follows:
\begin{equation} \label{eq:notation-scaled-RTZ-1}
	\cT^t_p = T_p \circ \dualoperator^{[pt]}, \ \hat R^t_p = R_p \circ \dualoperator^{[pt]} \ \text{ and } \cz^t_p = \cz_p \circ \vartheta^{[pt]}
\end{equation}
and
\begin{equation} \label{eq:notation-scaled-RTZ-2}
	T^t_p = T_p \circ \theta_{[pt]}, \ R^t_p = R_p \circ \theta_{[pt]} \ \text{ and } \cZ^t_p = \cZ_p \circ \theta_{[pt]}.
\end{equation}
Note that with these definitions, Theorem~\ref{thm:H-Pi} gives
\[ \H_p(t) = \left \lvert \Pi_p(t) \right \rvert = \Pi_p(t) \big( \cH_p(t)- \big) \ \text{ and } \ \Pi_p(t) = \cL \big( \hat R^t_p \mid_{\cH_p(t)} \big). \]

\subsection{Main assumptions and general tightness result}

Throughout the paper, we consider a double-sided L\'evy process $S_\infty = (S_\infty(t), t \in \R)$ of infinite variation, which is spectrally positive and does not drift to $+\infty$ as $t \to +\infty$. We denote by $\psi$ its Laplace exponent which is assumed to satisfy
\[ \int_1^\infty \frac{\d u}{\psi(u)} < \infty. \]
We let $T_\infty$ denote the ladder time process of $(S_\infty(t), t \geq 0)$ and $\cZ_\infty = S_\infty \circ T_\infty$ denote its ladder height process. We also consider $\cH_\infty$ its associated height process which is therefore the $\psi$-height process.

Duquesne and Le Gall~\cite{Duquesne02:0} proved that the following assumption implies the joint convergence of the Lukasiewicz path together with the genealogical height and contour processes. Proposition~\ref{prop:convergence-genealogy} below states a slight extension of this result needed for our purposes. Note that the condition on $Z_p$ below is automatically satisfied in the non-triangular case.

\begin{assumption} {\text{\textnormal{G}}'} \label{ass-G'}
	We say that Assumption~\ref{ass-G'} holds if $S_p \Rightarrow S_\infty$ with $S_\infty$ introduced above, and if for every $\delta > 0$ we have
	\[ \liminf_{p \to \infty} \ \P \left( Z_p([\delta / g_p]) = 0 \right) > 0 \]
	where $(Z_p(k), k \geq 0)$ is a Galton--Watson process with offspring distribution $\lvert \cP^*_p \rvert$ and started with $[p g_p]$ individuals.
\end{assumption}

Note that Assumption~\ref{ass-G'} implies that $p g_p \to \infty$. In addition to the genealogical Assumption~\ref{ass-G'}, we will need the next chronological assumptions.

\begin{assumption} {\text{\textnormal{C1'}}} \label{ass-C1'}
	We say that Assumption~\ref{ass-C1'} holds if $R_p \Rightarrow R_\infty$ with $R_\infty$ a non-degenerate subordinator. In this case, its Laplace exponent is denoted by $\phi$.
\end{assumption}

\begin{assumption} {\text{\textnormal{C2'}}} \label{ass-C2'}
	We say that Assumption~\ref{ass-C2'} holds if $\cV_p \Rightarrow \cV_\infty$ with $\cV_\infty$ a non-degenerate subordinator.
\end{assumption}

\begin{theorembis} {\ref{thm:tightness}'} [Triangular version of Theorem~\ref{thm:tightness}] \label{thm:tightness'}
	Under Assumptions~\ref{ass-G'},~\ref{ass-C1'} and~\ref{ass-C2'}, the sequence $(\Pi_p, \cV_p)$ is tight. Moreover, let $(\Pi_\infty, \cV_\infty)$ be any accumulation point and for $t \geq 0$ define $\H_\infty(t) = \lvert \Pi_\infty(t) \rvert$, $\C^\star_\infty(t) = \H_\infty \circ \cV_\infty^{-1}(t)$ and $\X_\infty(t) = t - \cV_\infty(\cV_\infty^{-1}(t)-)$. Then:
	\begin{enumerate}
		\item \label{tightness-1} $\Pi_\infty$ is almost surely continuous;
		\item for every $t \geq 0$, $\H_\infty(t)$ and $\C^\star_\infty(t)$ are almost surely finite and strictly positive;
		\item \label{tightness-3} if $(\Pi_p, \cV_p) \Rightarrow (\Pi_\infty, \cV_\infty)$ along a subsequence, then for any finite set $I \subset \R_+$ we have
		\begin{equation}
			\left(\Pi_p, \cV_p, (\H_p(t))_{t \in I}, (\C^\star_p(t))_{t \in I}, \X_p \right) \Rightarrow \left(\Pi_\infty, \cV_\infty, (\H_\infty(t))_{t \in I}, (\C^\star_\infty(t))_{t \in I}, \X_\infty \right) \tag{\ref{eq:full-conv}}
		\end{equation}
		along the same subsequence.
	\end{enumerate}
\end{theorembis}

The proof of Theorem~\ref{thm:tightness'} relies on genealogical results which we establish first.

\subsection{Genealogical convergence} 

For every $t\in\R$, define $S^t_\infty = \Theta_t(S_\infty)$ the shifted version of $S_\infty$ at time $t$, i.e.,
\[S_\infty^t(u) = S_\infty(t+u) - S_\infty(t), \ u \geq 0, \] 
and $\hat S^t_\infty$ the dual path of $S_\infty$ at time $t$ as
\[\hat S_\infty^t(u) = S_\infty(t) - S_\infty(t-u), \ u \geq 0.\] 
For every random variable $W$ that can be written as a measurable function of $S_\infty$, $\hat W^t$ or $W \circ \dualoperator^t$ will refer to the same functional applied to $\hat S_\infty^t$ and $W^t$ or $W \circ \theta_t$ will refer to the same functional applied to $S_\infty^t$. This gives for instance sense to $\cT^t_\infty, \hat \cZ^t_\infty$, $T^t_\infty$ and $\cZ^t_\infty$ that will appear repeatedly, and is coherent with the notation introduced in~\eqref{eq:notation-scaled-RTZ-1} and~\eqref{eq:notation-scaled-RTZ-2}.

An important relation, which follows from the definition of the height process, is that for every $t \geq 0$ we have $\cH_\infty(t) = (\hat T^t_\infty)^{-1}(t)$ almost surely. Note that this is the continuous analog of the relation $\cH(n) = T^{-1}(n) \circ \vartheta^n$ of~\eqref{eq:H-Pi}. The next result is an extension of a result due to Duquesne and Le Gall~\cite{Duquesne02:0}

\begin{prop} \label{prop:convergence-genealogy}
	If Condition~\ref{ass-G'} holds, then
	\begin{equation} \label{eq:joint-conv}
		\left ( S_p, \cH_p, (\cz^t_p, \cT^t_p, (\cT^t_p)^{-1})_{t \in \Q} \right) \ 
		\Rightarrow \ \left ( S_\infty, \cH_\infty, (\cz^t_\infty, \cT^t_\infty, (\cT^t_\infty)^{-1})_{t \in \Q} \right).
	\end{equation}
\end{prop}

\begin{proof}
	Throughout the proof, we repeatedly use the following token in order to get joint convergence results: if $(X_p, Y_p) \Rightarrow (X_\infty, Y_\infty)$ and $(X_p, Z_p) \Rightarrow (X_\infty, Z_\infty)$ with $Y_\infty$ and $Z_\infty$ measurable functions of $X_\infty$, then $(X_p, Y_p, Z_p) \Rightarrow (X_\infty, Y_\infty, Z_\infty)$. This comes from the continuous mapping theorem and will be uses without further mention.
	
	First, we reduce the proof of~\eqref{eq:joint-conv} to the simpler convergence
	\begin{equation} \label{eq:joint-conv-reduction}
		(S_p, T_p, T^{-1}_p, \cZ_p) \Rightarrow (S_\infty, T_\infty, T^{-1}_\infty, \cZ_\infty).
	\end{equation}
	Indeed, if~\eqref{eq:joint-conv-reduction} holds, then for each fixed $t \geq 0$ a simple time-reversal argument implies that $(S_p, \hat T^t_p, (\hat T^t_p)^{-1}, \hat \cZ^t_p) \Rightarrow (S_\infty, \hat T^t_\infty, (\hat T^t_\infty)^{-1}, \hat \cZ^t_\infty)$ and since $(\hat T^t_\infty, (\hat T^t_\infty)^{-1}, \hat \cZ^t_\infty)$ is measurable with respect to $S_\infty$ we obtain the joint convergence for $t \in \Q$. Moreover, Duquesne and Le Gall~\cite[Corollary $2.5.1$]{Duquesne02:0} have proved that Condition~\ref{ass-G'} implies that $(S_p, \cH_p) \Rightarrow (S_\infty, \cH_\infty)$. Since again $\cH_\infty$ is a measurable function of $S_\infty$, this finally gives the full joint convergence of~\eqref{eq:joint-conv}. The rest of the proof is therefore devoted to proving~\eqref{eq:joint-conv-reduction}.
	
	The joint convergence of the Lukasiewicz path together with its local time process at its maximum is proved in the proof of Theorem~$2.2.1$ in~\cite{Duquesne02:0}: to be precise, the almost sure convergence $\gamma_p^{-1} \Lambda^{(p)}_{[p \gamma_p t]} \to L_t$ stated in Equation~$(50)$ of~\cite{Duquesne02:0} implies with our notation that $(S_p, (T^{-1}_p(t))_{t \in I}) \Rightarrow (S_\infty, (T^{-1}_\infty(t))_{t \in I})$ for every finite set $I \subset \R_+$. Since $T^{-1}_p$ and $T^{-1}_\infty$ are non-decreasing and $T^{-1}_\infty$ is continuous (as the local time process at $0$ of $S_\infty$ reflected at its maximum), standard properties of the Skorohod topology imply that the finite-dimensional convergence $T^{-1}_p \fdd T^{-1}_\infty$ actually implies the functional convergence $T^{-1}_p \Rightarrow T^{-1}_\infty$, see for instance~\cite[Theorem VI.2.15]{Jacod03:0}. Thus we obtain $(S_p, T^{-1}_p) \Rightarrow (S_\infty, T^{-1}_\infty)$.
	
	Because $T_p$ for $p \in \N \cup \{\infty\}$ is non-decreasing, we have $(T^{-1}_p)^{-1} = T_p$ and so general properties of the inverse map in the Skorohod topology show that the convergence $(S_p, T^{-1}_p) \Rightarrow (S_\infty, T^{-1}_\infty)$ implies $(S_p, T^{-1}_p, (T_p(t))_{t \in I}) \Rightarrow (S_\infty, T^{-1}_\infty, (T_\infty(t))_{t \in I})$ for any finite set $I \subset \R_+$, see for instance the remark following Theorem~$7.1$ in~\cite{Whitt80:0}. Since $T_p$ is a random walk, we thus obtain $(S_p, T^{-1}_p, T_p) \Rightarrow (S_\infty, T^{-1}_\infty, T_\infty)$.
	
	Since $T_\infty$ is strictly increasing, this convergence implies that
	\[ \left( S_p, T^{-1}_p, T_p, (S_p \circ T_p(t))_{t \in I} \right) \Rightarrow \left( S_\infty, T^{-1}_\infty, T_\infty, (S_\infty \circ T_\infty(t))_{t \in I} \right) \]
	for any finite set $I \subset \R_+$, see for instance~\cite[Lemma~$2.3$]{Kurtz91:1}. Since $S_p \circ T_p = \cZ_p$ for $p \in \N \cup \{\infty\}$ and $\cZ_p$ is a random walk, the previous finite-dimensional result can be strengthened to a functional one and so we get the desired result.
\end{proof}

\subsection{Proof of Theorem {\protect\ref{thm:tightness'}}} \label{proof:tightness-general-CMJ}

We decompose the proof into several steps.

\subsection*{Step 1: proof of tightness and of~\eqref{tightness-1}}

By assumption, $\cV_p \Rightarrow \cV_\infty$ and so the proof of this step amounts to showing that $(\Pi_p)$ is tight (as a $\mathcal W$-valued process) and that any accumulation point is continuous.

The measure $\Pi(n)$ describes the ancestral line of the individual $n$, where the first atoms correspond to the ancestors closest to the root. In particular, we have $\Pi(n) \mid_{\cH(\mu)-1} = \Pi(m) \mid_{\cH(\mu)-1}$ with $\mu$ the most recent common ancestor of $m$ and $n$, whose height $\cH(\mu)$ is given by $\cH(\mu) = \min_{\{m, \ldots, n\}} \cH - \Indicator{\mu \neq m}$. A formal proof of these facts is given in~\cite{Schertzer:0}. Upon scaling we thus obtain $\Pi_p(s)\mid_h = \Pi_p(t) \mid_h$ with $h = \min_{[s,t]} \cH_p - 2 c_p$ and so from the expression~\eqref{eq:distance} of the distance $d$ we get that
\[ d\left(\Pi_p(s), \Pi_p(t)\right) \leq \left \lvert \cH_p(t) - \cH_p(s) \right \rvert + \cH_p(t)\wedge \cH_p(s) - \inf_{[s,t]} \cH_p + 2 c_p. \]
Since $\cH_p \Rightarrow \cH_\infty$ with $\cH_\infty$ continuous, the above inequality gives the tightness of $(\Pi_p)$ and also shows that any accumulation point is almost surely continuous.

\subsection*{Step 2: exchangeability argument}

We now present a simple yet crucial exchangeability argument. For $a < b$ let us write $\omega_a^b = (\omega_k, k = [a], \ldots, [b]-1)$. Then for any $t, k$, exchangeability implies that $\omega_0^k \circ \vartheta^{\cV^{-1}(t)}$ conditioned on $\cV^{-1}(t) \geq k$ is equal in distribution to $\omega_0^k$ conditioned on $\cV^{-1}(t) \geq k$. Indeed, since
\[ \E_p \left( \varphi(\omega_0^k) \mid \cV^{-1}(t) \geq k \right) = \frac{1}{\P_p(\cV^{-1}_p(t) \geq k)} \sum_{\ell \geq k} \E_p \left( \varphi(\omega_0^k) ; \cV^{-1}(t) = \ell \right) \]
in order to prove this claim it is enough to prove that
\[ \E_p \left( \varphi(\omega_0^k) ; \cV^{-1}(t) = \ell \right) = \E_p \left( \varphi(\omega_0^k) \circ \vartheta^\ell ; \cV^{-1}(t) = \ell \right) \]
for every $\ell \geq k$. To see this, write
\[ \E_p \left( \varphi(\omega_0^k) ; \cV^{-1}(t) = \ell \right) = \E_p \left( \varphi(\omega_0^k) ; \sum_{i=0}^{\ell-1} V_i \leq t < \sum_{i=0}^\ell V_i \right) = \E_p \left( \psi(\omega_0, \ldots, \omega_\ell) \right) \]
with
\[ \psi(\omega_0, \ldots, \omega_\ell) = \varphi(\omega_0^k) \Indicator{t - V_i <  \sum_{i=0}^{\ell-1} V_i \leq t}. \]
Since the $\omega_k$'s are i.i.d., we have
\begin{align*}
	\E_p \left( \psi(\omega_0, \ldots, \omega_{\ell-1}, \omega_\ell) \right) & = \E_p \left( \psi(\omega_{\ell-1}, \ldots, \omega_0, \omega_\ell) \right)\\
	& = \E_p \left( \varphi(\omega_{\ell-1}, \ldots, \omega_{\ell-k}) ; \sum_{i=0}^{\ell-1} V_i \leq t < \sum_{i=0}^\ell V_i \right)\\
	& = \E_p \left( \varphi(\omega_0^k) \circ \vartheta^\ell ; \cV^{-1}(t) = \ell \right)
\end{align*}
This exchangeability property will be used in the sequel in the following form: for any measurable mapping $\varphi \geq 0$ acting on finite sequence of sticks, we have
\begin{equation} \label{eq:exchangeability}
	\E_p \left( \varphi(\omega_0^{\varepsilon p}) \circ \vartheta^{p \cV^{-1}_p(t)} ; \cV^{-1}_p(t) \geq \varepsilon \right) = \E_p \left( \varphi(\omega_0^{\varepsilon p}) ; \cV^{-1}_p(t) \geq \varepsilon \right).
\end{equation}

\subsection*{Step 3: proof of~\eqref{tightness-3} and that $\H_\infty(t), \C^\star_\infty(t) < \infty$} Assume without less of generality that $(\Pi_p, \cV_p) \Rightarrow (\Pi_\infty, \cV_\infty)$. By standard properties of the Skorohod topology, this implies that $(\Pi_p, \cV_p, \X_p) \Rightarrow (\Pi_\infty, \cV_\infty, \X_\infty)$ and so in order to prove~\eqref{tightness-3} and that $\H_\infty(t)$ and $\C^\star_\infty(t)$ are almost surely finite, it is enough to prove that
\[ (\Pi_p, \cV_p, \H_p(\tau_p)) \Rightarrow (\Pi_\infty, \cV_\infty, \H_\infty(\tau_\infty)) \ \text{ for } \ \tau_p = t \ \text{ or } \ \tau_p = \cV^{-1}_p(t), \ t \geq 0, p \in \N \cup \{\infty\} \]
with $\H_\infty(\tau_\infty)$ being almost surely finite. We will only show that
\begin{equation}\label{eq:X}
	\H_p(\tau_p) \Rightarrow \H_\infty(\tau_\infty) \ \text{ with } \ \P(\H_\infty(\tau_\infty) < \infty) = 1
\end{equation}
where either $\tau_p = t$, or $\tau_p = \cV^{-1}_p(t)$ ($p \in \N \cup \{\infty\}$, $t > 0$). Indeed, the joint convergence with $(\Pi_p,\cV_p)$ relies on the same arguments but is only notationally more cumbersome. In order to show~\eqref{eq:X}, we will need the next technical lemma. 

\begin{lemma} \label{lemma:conv-terminal-values}
	Let $t_p \to t_\infty$ in $\R_+$, $w_p$ for $p \in \N \cup \{\infty\}$ a $\mathcal{W}$-valued process with $w_p(t)$ non-decreasing for every $t \geq 0$, and $\zeta_p$ the life time of $w_p(t_p)$. If $w_p \to w_\infty$, $w_\infty$ is continuous and
	\begin{equation} \label{eq:cond-terminal-values}
		\lim_{c \uparrow \zeta_\infty} \limsup_{p \to \infty} \left ( w_p(t_p)(\zeta_p-) - w_p(t_p)(c) \right ) = 0
	\end{equation}
	then $w_\infty(t_\infty)(\zeta_\infty-) < \infty$ and $w_p(t_p)(\zeta_p-) \to w_\infty(t_\infty)(\zeta_\infty-)$.
\end{lemma}

\begin{proof}
	Let $c < \zeta_\infty$ with $w_\infty(t_\infty)$ continuous at $c$: then $w_p(t_p) \to w_\infty(t_\infty)$ in $\mathcal{W}$ (which holds because $w_p \to w_\infty$ with $w_\infty$ continuous) implies $w_p(t_p)(c) \to w_\infty(t_\infty)(c)$ (note that $c < \zeta_p$ for $p$ large enough as $w_p \to w_\infty$ implies $\zeta_p \to \zeta_\infty$). Since $w_\infty(t_\infty)(c) < \infty$, this implies in view of~\eqref{eq:cond-terminal-values} that $\limsup_{p \to \infty} w_p(t_p)(\zeta_p-) < \infty$. Next, since $w_p(t_p)(\zeta_p-) \geq w_p(t_p)(c)$ for $p$ large enough, we obtain
	\[ w_\infty(t_\infty)(c) \leq \limsup_{p \to \infty} w_p(t_p)(\zeta_p-) \]
	and then
	\[ w_\infty(t_\infty)(\zeta_\infty-) \leq \limsup_{p \to \infty} w_p(t_p)(\zeta_p-) \]
	by letting $c \uparrow \zeta_\infty$ along continuity points. This shows that $w_\infty(t_\infty)(\zeta_\infty-) < \infty$ and we now proceed to showing that $w_p(t_p)(\zeta_p-) \to w_\infty(t_\infty)(\zeta_\infty-)$. For $p \in \N \cup \{\infty\}$ and $c < \zeta_\infty$ let $\delta_p(c) = \left \lvert w_p(t_p)(\zeta_p-) - w_p(t_p)(c) \right \rvert$: then
	\[ \left \lvert w_p(t_p)(\zeta_p-) - w_\infty(t_\infty)(\zeta_\infty-) \right \rvert \leq \delta_p(c) + \delta_\infty(c) + \left \lvert w_p(t_p)(c) - w_\infty(t_\infty)(c) \right \rvert. \]
	Proceeding with similar arguments as above we obtain the result by letting first $p \to \infty$ and then $c \uparrow \zeta_\infty$.
\end{proof}

Since $\H_p(\tau_p) = \Pi_p(t) (\cH_p(\tau_p)-)$, proving~\eqref{eq:X} is the same as proving
\begin{equation}\label{eq:tt}
	\Pi_p(\tau_p) (\cH_p(\tau_p)-) \Rightarrow \Pi_\infty(\tau_\infty) (\cH_\infty(\tau_\infty)-) \ \text{ with } \ \P \left( \Pi_\infty(\tau_\infty) (\cH_\infty(\tau_\infty)-) < \infty \right) = 1.
\end{equation}
By standard arguments (e.g., Skorohod's representation theorem), Lemma~\ref{lemma:conv-terminal-values} implies that in order to prove this, it is enough to prove that
\[ \limsup_{p \to \infty} \ \P_p \Big( \Pi_p(\tau_p) (\cH_p(\tau_p)-) - \Pi_p(\tau_p) (\cH_p(\tau_p)-c) \geq \eta \Big) \mathop{\longrightarrow}_{c \to 0} 0 \]
for every $\eta > 0$. Since $\Pi(n) = \cL(\hat R^n \mid_{\cH(n)})$ according to Theorem~\ref{thm:H-Pi}, we have
\[ \Pi_p(\tau_p) (\cH_p(\tau_p)-) - \Pi_p(\tau_p) (\cH_p(\tau_p)-c) \leq \hat R^{\tau_p}_p(c) \]
which reduces the proof of~\eqref{eq:tt} to showing that
\begin{equation} \label{eq:goal-R}
	\limsup_{p \to \infty} \ \P \left( \hat R^{\tau_p}_p(c) \geq \eta \right) \mathop{\longrightarrow}_{c \to 0} 0
\end{equation}
for every $\eta > 0$. If $\tau_p = t$, then $\hat R^{\tau_p}_p(c)$ is identical in law to $R_p(c)$, and so~\eqref{eq:goal-R}
follows by the convergence of $R_p$ and the continuity of $R_\infty$ at $0$.

Let us now prove~\eqref{eq:goal-R} for $\tau_p = \cV^{-1}_p(t)$. Since $T^{-1}_p(T_p(t)) = t$ and $R_p \circ T^{-1}_p$ is increasing, in the event $\{T_p(c) \leq \varepsilon\}$ we have
\[ R_p(c) = R_p \circ T^{-1}_p (T_p(c)) \leq R_p \circ T^{-1}_p (\varepsilon). \]
In particular,
\begin{align*}
	\P_p \left( R_p(c) \circ \vartheta^{\tau_p} \geq \eta \right) & \leq \P_p \left( R_p \circ T^{-1}_p(\varepsilon) \circ \vartheta^{\tau_p} \geq \eta \right) + \P_p \left( T_p(c) \circ \vartheta^{\tau_p} \geq \varepsilon \right)\\
	& \leq \P_p \left( R_p \circ T^{-1}_p(\varepsilon) \circ \vartheta^{\tau_p} \geq \eta, \cV^{-1}_p(t) \geq \varepsilon \right)\\
	& \hspace{20mm} + \P_p \left( T_p(c) \circ \vartheta^{\tau_p} \geq \varepsilon, \cV^{-1}_p(t) \geq \varepsilon \right)\\
	& \hspace{20mm} + 2 \P_p \left( \cV^{-1}_p(t) \leq \varepsilon \right).
\end{align*}
By definition, we can write $R_p \circ T^{-1}_p(\varepsilon)$ as a function of $\omega_0^{p \varepsilon}$, i.e., $R_p \circ T^{-1}_p(\varepsilon) = \varphi(\omega_0^{p \varepsilon})$ for some measurable mapping $\varphi$. The exchangeability identity~\eqref{eq:exchangeability} thus gives
\[ \P_p \left( R_p \circ T^{-1}_p(\varepsilon) \circ \vartheta^{\tau_p} \geq \eta, \cV^{-1}_p(t) \geq \varepsilon \right) = \P_p \left( R_p \circ T^{-1}_p(\varepsilon) \geq \eta, \cV^{-1}_p(t) \geq \varepsilon \right) \]
and, for the same reasons,
\[ \P_p \left( T_p(c) \circ \vartheta^{\tau_p} \geq \varepsilon, \cV^{-1}_p(t) \geq \varepsilon \right) = \P_p \left( T_p(c) \geq \varepsilon, \cV^{-1}_p(t) \geq \varepsilon \right) \]
which thus leads to the bound
\[ \P_p \left( R_p(c) \circ \vartheta^{\tau_p} \geq \eta \right) \leq \P_p \left( R_p \circ T^{-1}_p(\varepsilon) \geq \eta \right) + \P_p \left( T_p(c) \geq \varepsilon \right) + 2 \P_p \left( \cV^{-1}_p(t) \leq \varepsilon \right). \]
Letting first $p \to \infty$, then $c \to 0$ and finally $\eta \to 0$ yields~\eqref{eq:goal-R} in the case $\tau_p = \cV^{-1}_p(t)$, which achieves the proof of this step.

\subsection*{Step 4: proof that $\H_\infty(t), \C^\star_\infty(t) > 0$}

We now prove that $\H_\infty(t)$ and $\C^\star_\infty(t)$ are almost surely strictly positive. For $\H_\infty(t)$, this comes from the convergence $\H_p(t) \Rightarrow \H_\infty(t)$ which we have just proved, together with the fact that $\H_p(t)$ is equal in distribution to $R_p \circ \bTm_p(t)$ with $R_p$ and $\bTm_p$ converging to two subordinators (this does not imply that $R_p \circ \bTm_p(t) \Rightarrow R_\infty \circ T^{-1}_\infty(t)$, but it does imply that any accumulation point is necessarily almost surely $>0$).

Let us now prove that
\[ \liminf_{p \to \infty} \ \P_p \left( \C^\star_p(t) \geq \eta \right) \mathop{\longrightarrow}_{\eta \downarrow 0} 1. \]
We use the same exchangeability arguments as in the previous step. We have
\[ \C^\star_p(t) = \H_p(\cV^{-1}_p(t)) = \left( R_p \circ \bTm_p \circ \vartheta^{p \cV^{-1}_p(t)} \right) (\cV^{-1}_p(t)). \]
Let $\varphi_p$ be the measurable mapping such that $\varphi_p(\omega_0^n) = R_p \circ \bTm_p(n/p)$, so that $\varphi_p(\omega_0^{pt}) = R_p \circ \bTm_p(t)$: then for any $u \geq t$, in the event $\{ \cV^{-1}_p(t) = u \}$ we have
\[ \left( R_p \circ \bTm_p \circ \vartheta^{p \cV^{-1}_p(t)} \right) (\cV^{-1}_p(t)) = \left( R_p \circ \bTm_p \circ \vartheta^{p u} \right) (u) = \varphi_p(\omega_0^{pu}) \circ \vartheta^{p u} \geq \varphi_p(\omega_0^{pt}) \circ \vartheta^{p u} \]
with the last inequality following from the monotonicity of $n \mapsto \varphi_p(\omega_0^n)$. In particular,
\[ \P_p \left( \C^\star_p(t) \geq \eta \right) \geq \P_p \left( \H_p(\cV^{-1}_p(t)) \geq \eta, \cV^{-1}_p(t) \geq \varepsilon \right) \geq \P_p \left( \varphi_p(\omega_0^{pt}) \circ \vartheta^{p \cV^{-1}_p(t)} \geq \eta, \cV^{-1}_p(t) \geq \varepsilon \right) \]
and~\eqref{eq:exchangeability} finally gives
\[ \P_p \left( \C^\star_p(t) \geq \eta \right) \geq \P_p \left( \varphi_p(\omega_0^{pt}) \geq \eta, \cV^{-1}_p(t) \geq \varepsilon \right) \geq \P_p \left( \H_p(t) \geq \eta \right) - \P_p \left( \cV^{-1}_p(t) \leq \varepsilon \right). \]
Letting first $p \to \infty$, then $\eta \downarrow 0$ and finally $\varepsilon \to 0$ gives the result, since we know that $\H_p(t) \Rightarrow \H_\infty(t)$ with $\P(\H_\infty(t) > 0) = 1$.

\section{Convergence of the height process} \label{sub:proof-of-thm:main-result-2}

In this section we state and prove the following triangular version of Theorem~\ref{thm:height}.

\begin{theorembis} {\ref{thm:height}'} [Triangular version of Theorem~\ref{thm:height}] \label{thm:height'}
	If Assumptions~\ref{ass-G'},~\ref{ass-C1'} and~\ref{ass-C2'} are satisfied and the asymptotic independence relation
	\begin{equation} \label{eq:independence-condition'}
		\P_p \left( c_p R(1) \geq \varepsilon \mid T(1) \geq \delta p \right) \to 0 \ \text{ for every } \ \varepsilon, \delta > 0 \tag{\ref{eq:independence-condition}'}
	\end{equation}
	holds, then $\Pi_p \Rightarrow \Pi_\infty$ where $\Pi_\infty$ is the $\psi/\phi$ snake. In particular, $\H_p \fdd \lvert \Pi_\infty(\, \cdot \,) \rvert$.
\end{theorembis}

The rest of the section is devoted to the proof of Theorem~\ref{thm:height'}: until the end of this section, we thus assume that all the assumptions of Theorem~\ref{thm:height'} hold. The core of the proof is given in this section, and the lengthy proof of a technical result (Proposition~\ref{prop:more-ind}) is in Section~\ref{sect:preliminaries}.

First, we note that the assumptions made imply that $(R_p, T_p, \cZ_p) \Rightarrow (R_\infty, T_\infty, \cZ_\infty)$ with $R_\infty \ci (\cZ_\infty, T_\infty)$. Indeed, $R_p \Rightarrow R_\infty$ by assumption and $(T_p, \cZ_p) \Rightarrow (T_\infty, \cZ_\infty)$ by Proposition~\ref{prop:convergence-genealogy}. Moreover, the independence condition~\eqref{eq:independence-condition'} means precisely that in the limit, $R_\infty$ and $T_\infty$ do not jump simultaneously, and as any accumulation point of $(R_p, T_p, \cZ_p)$ is necessarily a subordinator, this implies that $R_\infty \ci T_\infty$ from which we get $R_\infty \ci (T_\infty, \cZ_\infty)$.

In order to prove Theorem \ref{thm:height'}, we will need the following stronger independence property. The proof is rather long and technical, and thus postponed to Section~\ref{sect:preliminaries}.

\begin{prop}\label{prop:more-ind}
	Assume that the assumptions of Theorem~\ref{thm:height'} hold. Then for any finite set $J\subset \R$ we have $R_\infty \ci (\cZ_\infty^t, T_\infty^t)_{t\in J}$.
\end{prop}

\begin{remark}
	In Theorem~\ref{thm:extended-BH}, we introduce a simple extension of a Bellman--Harris forest. In this case, one can check that $R_p \ci (T_p, \cZ_p)$ but that $R_p$ is in general not independent from $(\cZ_p^t, T_p^t)_{t\in J}$ for $\lvert J \rvert>1$. As a consequence, the passage from $R_\infty \ci (T_\infty, \cZ_\infty)$ to Proposition~\ref{prop:more-ind} only holds asymptotically and actually requires some delicate arguments developed in Section~\ref{sect:preliminaries}.
\end{remark}

Recall from Theorem \ref{thm:tightness'} that $(\Pi_p)$ is tight. Since by assumptions we have $(S_p, \cH_p) \Rightarrow (S_\infty, \cH_\infty)$ and $(R_p, T_p, \cZ_p) \Rightarrow (R_\infty, T_\infty, \cZ_\infty)$ with $R_\infty \ci (T_\infty, \cZ_\infty)$, it follows by duality and the previous result that the sequence
\[ \left ( S_p, \cH_p, \Pi_p, (\cz^t_p, \cT^t_p, \hat R^t_p)_{t \in \Q}\right) \]
is tight. Consider any accumulation point and assume without loss of generality in the rest of this section that there is almost sure convergence
\begin{equation} \label{eq:big-conv}
	\left ( S_p, \cH_p, \Pi_p, (\cz^t_\infty, \cT^t_p, \hat R^t_p)_{t \in \Q} \right)
	\to \left ( S_\infty, \cH_\infty, \Pi_\infty, (\cz^t_\infty, \cT^t_\infty, \hat R^t_\infty)_{t \in \Q}\right)
\end{equation}
with $\hat R^t_\infty \ci (\cz^t_\infty, \cT^t_\infty)$ for any $t \in \Q$ and $\Pi_\infty$ continuous with life-time process $\cH_\infty$. Since
\begin{equation} \label{eq:formula-T-H}
	(\hat T^t_\infty)^{-1}(t-s) = \cH_\infty(t) - \min_{[s,t]} \cH_\infty, \ 0 \leq s \leq t \in \Q
\end{equation}
(see for instance~\cite[Lemma~$5.3$]{Schertzer:0}) we obtain in particular $\hat R^t_\infty \ci \cH_\infty(t)$. Actually, 
as a consequence of Proposition \ref{prop:more-ind}, a stronger result holds.

\begin{lemma}
	For any $t \in \Q$ we have $\hat R^t_\infty \ci \cH_\infty$.
\end{lemma}

\begin{proof}
	According to Proposition~\ref{prop:more-ind} we have $R_\infty \ci (T_\infty^u, u \in J)$ for any finite set $J \subset \R$, which implies $\hat R^t_\infty \ci (T_\infty^u \circ \dualoperator^t, u \in J)$ for every $t \in \Q$. Since $T^u_\infty \circ \dualoperator^t = \hat T^{t-u}_\infty$ (i.e., reverting time from $t$ and then shifting by $u$ is the same as reverting time from $t-u$) and $\cH_\infty(u) = (\hat T^u_\infty)^{-1}(u)$ we obtain the result.
\end{proof}

\begin{lemma} \label{lemma:formula-Pi-R}
	For any $t \in \Q$ we have $\Pi_\infty(t) = \cL(\hat R^t_\infty \mid_{\cH_\infty(t)})$.
\end{lemma}

\begin{proof}
	Since $\hat R^t_\infty \ci \cH_\infty(t)$ and $\hat R^t_\infty$ is a subordinator, $\hat R^t_\infty$ does not jump at time $\cH_\infty(t)$ and so Lemma~\ref{lemma:continuity-properties} and the convergence~\eqref{eq:big-conv} imply that
	\[ \Pi_p(t) = \cL \left( \hat R^t_p \mid _{\cH_p(t)} \right) \to \cL \left( \hat R^t_\infty \mid _{\cH_\infty(t)} \right) \]
	with the convergence holding in the usual Skorohod space. On the other hand, since $\Pi_p \to \Pi_\infty$ with $\Pi_\infty$ continuous, we have also $\Pi_p(t) \to \Pi_\infty(t)$, this time in $\mathcal{W}$. From these two convergences one gets the desired result.
\end{proof}

For $s \leq t$ with $t \in \Q$, we define
\[ \tilde \Pi_\infty(s,t) = \cL \big( \hat R^t_\infty \mid_{(\hat T^t_\infty)^{-1}(t-s)} \big).\]
Intuitively, $\tilde \Pi_\infty(s,t)$ is the section of the chronological spine from $t$ that does not overlap the 
chronological spine from $s$.

\begin{lemma} \label{lemma:snake-like-2}
	Let $s \leq t$ with $t \in \Q$. Then $\Pi_\infty(t) = \left[\Pi_\infty(s), \ \tilde \Pi_\infty(s,t) \right]_{\min_{[s,t]}\cH_\infty}$. Moreover, conditionally on $\cH_\infty$, $\tilde \Pi_\infty(s,t)$ is a subordinator with Laplace exponent $\phi$ killed at $(\hat T^t_\infty)^{-1}(t-s)$ and independent from $(\Pi_\infty(u), u \leq s)$.
\end{lemma}

\begin{proof}
	That conditionally on $\cH_\infty$, $\tilde \Pi_\infty(s,t)$ is a subordinator with Laplace exponent $\phi$ killed at $(\hat T^t_\infty)^{-1}(t-s)$ is immediate from its definition and the fact that $\hat R^t_\infty$ is a subordinator with Laplace exponent $\phi$ independent from $\cH_\infty$. It remains to prove $\Pi_\infty(t) = \left[\Pi_\infty(s) , \ \tilde \Pi_\infty(s,t) \right]_{\min_{[s,t]}\cH_\infty}$ and the conditional independence.
	\\
	
	\noindent \textit{Proof of the relation $\Pi_\infty(t) = \left[\Pi_\infty(s), \ \tilde \Pi_\infty(s,t) \right]_{\min_{[s,t]}\cH_\infty}$.} We have seen in the proof of Theorem~\ref{thm:tightness'} that $\Pi_p(t) \mid_{\min_{[s,t]} \cH_p-2c_p} = \Pi_p(s) \mid_{\min_{[s,t]} \cH_p-2c_p}$: letting $p \to \infty$ we thus obtain $\Pi_\infty(t) \mid_{\min_{[s,t]} \cH_\infty} = \Pi_\infty(s) \mid_{\min_{[s,t]} \cH_\infty}$ (using as before the convergence of the $\hat R^t_p$'s, their asymptotic independence with $\cH_\infty$ and Lemma~\ref{lemma:continuity-properties}) and so
	\[ \Pi_\infty(t) = \left[ \Pi_\infty(t), \Theta_{\min_{[s,t]} \cH_\infty} (\Pi_\infty(t)) \right]_{\min_{[s,t]} \cH_\infty} = \left[ \Pi_\infty(s), \Theta_{\min_{[s,t]} \cH_\infty} (\Pi_\infty(t)) \right]_{\min_{[s,t]} \cH_\infty}. \]
	Since $\Pi_\infty(t) = \cL(\hat R^t_\infty \mid_{(\hat T^t_\infty)^{-1}(t)})$ (Lemma~\ref{lemma:formula-Pi-R}) and $(\hat T^t_\infty)^{-1}(t-s) = \cH_\infty(t) - \min_{[s,t]} \cH_\infty$ (Relation~\eqref{eq:formula-T-H}) we obtain
	\[ \Theta_{\min_{[s,t]} \cH_\infty} \left( \Pi_\infty(t) \right) = \cL \big( \hat R^t_\infty \mid_{(\hat T^t_\infty)^{-1}(t-s)} \big). \]
	Combining the last two displays proves the desired relation.
	\\
	
	\noindent \textit{Proof of the conditional independence.} We now prove that $\tilde \Pi_\infty(s,t)$ is independent of $(\Pi_\infty(u), u \leq s)$ conditionally on $\cH_\infty$. For finite $p$, define 
	\[\tilde \Pi_p(s,t) = \ \cL(\hat R^t_p\mid_{(\hat T^t_p)^{-1}(t-s)-}).\]
	 We have $(\tilde \Pi_p(s,t), \Theta_s(S_p)) \in m \cF_{\geq [ps]}$ and $((\hat R^u_p)_{u \leq s}, S_p\mid_s) \in m \cF_{< [ps]}$ which implies by letting $p \to \infty$ that
	\[ \left( \tilde \Pi_\infty(s,t), \Theta_s(S_\infty) \right) \ci \left( (\hat R^u_\infty)_{u \leq s, u \in \Q}, S_\infty\mid_s \right). \]
	Since $\Pi_\infty(u) = \cL(\hat R^u_\infty\mid_{\cH_{\infty}(u)})$ for $u \in \Q$, this gives
	\[ \left( \tilde \Pi_\infty(s,t), \Theta_s(S_\infty) \right) \ci \left( \left(\Pi_\infty(u) \right)_{u \leq s}, S_\infty \mid_s \right) \]
	(note that we can remove the condition $u \in \Q$ by continuity of $\Pi_\infty$) which further gives the independence between $\tilde \Pi_\infty(s,t)$ and $(\Pi_\infty(u), u \leq s)$ conditionally on $S_\infty$. Since finally $\sigma(S_\infty) = \sigma(\cH_\infty)$ this gives the desired conditional independence result.
\end{proof}

We can now prove Theorem~\ref{thm:height'}.

\begin{proof} [Proof of Theorem~\ref{thm:height'}]
	We already know that its life-time process is $\cH_\infty$ and is the height process associated to $S_\infty$. Thus according to Theorem~\ref{teo:exist-unique-snake} we only have to show that conditionally on $\cH_\infty$, $\Pi_\infty$ is (time-inhomogeneous) Markovian and satisfies for every $s \leq t$ the equality in distribution
	\[ \Pi_\infty(t) = \left[ \Pi_\infty(s) , \Gamma \right]_{\min_{[s,t]} \cH_\infty} \]
	with $\Gamma$ an independent subordinator with Laplace exponent $\phi$ killed at $\cH_\infty(t) - \min_{[s,t]} \cH_\infty$. To do so, we only have to prove that
	\[ \E \left[ f(\Pi_\infty(t)) \mid \cH_\infty, \Pi_\infty(u), u \leq s \right] = \E \left[ f \left( \left[ \Pi_\infty(s), \Gamma \right]_{\min_{[s,t]} \cH_\infty} \right) \mid \cH_\infty, \Pi_\infty(s) \right]. \]
	For $t \in \Q$ this is a direct consequence of Lemma~\ref{lemma:snake-like-2}, and so the result follows by continuity of $\Pi_\infty$.
\end{proof}

\section{Convergence of the contour process} \label{sec:general-contour}\label{sec:contour}

In this section we state and prove the following triangular version of Theorem~\ref{thm:contour}.

\begin{theorembis} {\ref{thm:contour}'} [Triangular version of Theorem~\ref{thm:contour}] \label{thm:contour'}
	If Assumptions~\ref{ass-G'},~\ref{ass-C1'} and~\ref{ass-C2'} are satisfied, if~\eqref{eq:independence-condition'} holds and if
	\begin{equation} \label{eq:independence-condition-2'}
		\P \left( \lvert \cP^*_p \rvert \geq \varepsilon p g_p \mid v_p V^*_p \geq \delta \right) \mathop{\longrightarrow}_{p \to \infty} 0 \ \text{ for every } \ \varepsilon, \delta > 0, \tag{\ref{eq:independence-condition-2}'}
	\end{equation}
	then $\cV_\infty$ is independent of $\H_\infty$ and $c_p / v_p \to \infty$. In particular, $\C^\star_p \fdd \H_\infty \circ \cV^{-1}_\infty$.
\end{theorembis}

Under the assumptions of this theorem, we know by Theorem~\ref{thm:height'} that $\Pi_p \Rightarrow \Pi_\infty$ and since $\cV_p \Rightarrow \cV_\infty$, we can assume without loss of generality that $(\Pi_p, \cV_p) \Rightarrow (\Pi_\infty, \cV_\infty)$ with an unknown (at this point) correlation structure between $\Pi_\infty$ and $\cV_\infty$.

As we will see, the main idea is that the asymptotic behavior of $\cV_p$ and $\C_p$ is governed by long edges, i.e., edges with length $\geq \varepsilon / v_p$ for some $\varepsilon > 0$. In contrast, $\H$ and $\C^\star$ only ``see'', by construction, the birth times of individuals, which are of the order of $1/c_p$ and are thus somehow insensitive to long edges because, intuitively, these long edges are close to the leaves.

To formalize this idea, for each $\varepsilon > 0$ and $k \geq 0$ we consider $\lambda_k$ the $k$th individual with an edge longer than $\varepsilon / v_p$ and $\lambda^+_k$ its last child: $\lambda_0 = 0$ and for $k \geq 1$,
\[ \lambda_k = \min \left\{ n > \lambda_{k-1}: V_n \geq \varepsilon / v_p \right\} \ \text{ and } \ \lambda^+_k = \inf \big\{ n \geq \lambda_k: S(n) = S(\lambda_k)-1 \big\}. \]
In particular, $\{\lambda_k, \ldots, \lambda^+_k-1\}$ is the set of $\lambda_k$'s descendants (including $\lambda_k$) and
\[ \bigcup_{k \geq 0} \left \{ \lambda_k, \ldots, \lambda^+_k-1 \right \}. \]
is the set of individuals with an ancestor (including potentially the individual herself) with an edge $\geq \varepsilon / v_p$. Note that the sequences $(\lambda_k)$ and $(\lambda^+_k)$ depend on $p$ and $\varepsilon$ but this dependency is omitted from the notation for simplicity. We will also denote with $\tilde{\cdot}$ scaled versions of these random times, namely
\[ \tilde \lambda_k^+ = \frac{1}{p} \lambda_k^+ \ \text{ and } \ \tilde \lambda_k^+ = \frac{1}{p} \lambda_k^+ \]
We finally define $\lambda^\infty_k$ as the time of the $k$th jump $\geq \varepsilon$ of $\cV_\infty$. Note that since $\cV_p \Rightarrow \cV_\infty$, standard properties of the Skorohod topology imply that
\begin{equation} \label{eq:conv-time}
	\left( \tilde \lambda_k, k \geq 0 \right) \Rightarrow \left( \lambda^\infty_k, k \geq 0 \right),
\end{equation}
see for instance~\cite[Proposition VI.2.17]{Jacod03:0}.

We will thus consider sequences $(X_{p, \varepsilon})$ indexed by the scaling parameter $p$ and also by $\varepsilon$, and for such sequences we will use the notation $X_{p, \varepsilon} \Rightarrow_{p, \varepsilon} X$ to mean that $X_{p, \varepsilon}$ converges weakly to $X$ when we let first $p \to \infty$ and then $\varepsilon \downarrow 0$, or more formally that for every bounded continuous function $f$ we have
\[ \limsup_{p \to \infty} \left \lvert \E \left[ f(X_{p, \varepsilon}) \right] - \E \left[ f(X) \right] \right \rvert \mathop{\longrightarrow}_{\varepsilon \to 0} 0. \]
For instance, standard properties of the Skorohod topology imply that $\cV_p^{< \varepsilon}(t) \Rightarrow_{p, \varepsilon} \texttt{d}t$ for each fixed $t \geq 0$, where $\texttt{d}\geq0$ denotes the drift of the process $\cV^\infty$ and $\cV_p^{< \varepsilon}$ for $p \in \N \cup \{\infty\}$ is obtained from $\cV_p$ by only keeping jumps $< \varepsilon$, i.e., by removing all jumps $\geq \varepsilon$.

\subsection{Step 1: $c_p / v_p \to \infty$}

The proof of $c_p / v_p \to \infty$ relies on the next lemma which shows that the time spent visiting subtrees rooted at long edges is negligible.

\begin{lemma} \label{lemma:subtrees-long-edges}
	For every $k \geq 0$ we have $\tilde \lambda^+_k - \tilde \lambda_k \Rightarrow 0$.
\end{lemma}

\begin{proof}
	We note that $\lambda^+_k - \lambda_k$ is equal in distribution to the hitting time of $-M$ by $S$, where $M \ci S$ is distributed as $\lvert \cP_{\lambda_1} \rvert$. Since $S_p \Rightarrow S_\infty$, the time needed for $S$ to hit $-\varepsilon p g_p$ is of the order of $p$. Since $\lvert \cP_{\lambda_1} \rvert / (p g_p) \Rightarrow 0$ by assumption, i.e., $M$ is negligible compared to $p g_p$, the hitting time of $M$ is also negligible compared to $p$, hence the result.
\end{proof}

We now prove that $c_p / v_p \to \infty$: by working along appropriate subsequences assume without loss of generality that $v_p / c_p \to \ell = \limsup_p(v_p / c_p)$. Since $R(k) \leq V_{T(k)-1}$ we have
\[ \H([pt]) = \left(\sum_{k: T(k)\leq [pt]} R(k) \right)\circ \vartheta^{[pt]} \leq \left(\sum_{k: T(k)\leq [pt]} V_{T(k)-1} \right)\circ \vartheta^{[pt]}. \]
Let
\[ E_p = \left\{ [pt] \notin \bigcup_{k \geq 1} \left\{ \lambda_k, \ldots, \lambda^+_k-1 \right\} \right\} \]
be the event that no ancestor of $[pt]$ has long edges. In $E_p$ we have
\[ \H([pt]) \leq \sum_{j\leq [pt]} V_j \Indicator{V_j<\eps / v_p} \]
and so upon scaling, we get $v_p \H_p(t) / c_p \leq \cV^{< \varepsilon}_p(t)$ in $E_p$. Since $\H_p(t) \Rightarrow R_\infty(T^{-1}_\infty(t))$, $\cV^{< \varepsilon}_p(t) \Rightarrow_{p, \varepsilon} \texttt{d} t$ and $\P(E_p) \to 1$ as a consequence of Lemma~\ref{lemma:subtrees-long-edges}, we obtain that $\P(\ell R_\infty(T^{-1}_\infty(t)) \leq \texttt{d}t) = 1$ which can only hold if $\ell = 0$.

\subsection{Step 2: snake-like property}

Let in the sequel
\[ X(m) = \sup_{\{0,\ldots,m\}} S - S(m). \]
In the following statement, we consider a functional $W$ whose $k$th increment is a functional of the sequence of sticks between the $k$th and $(k+1)$th record time of the path $S$. By standard fluctuation theory, $W$ can be seen as a renewal process
whose time is measured in local time units for the reflected process $X$.

The basic idea behind the next relation consists in decomposing the process $W$ before and after the local time corresponding to the excursion of $X$ straddling $t$. We refer the reader to Figure \ref{fig:discrete-snake}.

\begin{lemma}[Snake-like property]\label{prop:discrete-snake}
	Let $Z \in m \cF_{< T(1)}$ and 
	\begin{equation} \label{eq:W}
		W = \sum_{k \in \N} Z \circ \theta_{T(k)} \ \epsilon_{k+1}.
	\end{equation}
	Then $W \mid_{\tilde T^{-1}(m)} \in \cF_{\leq m}$ and if $m$ is not a weak record time for the walk $S$, then
	\begin{equation}\label{eq:discrete-snake}
		\Theta_{T^{-1}( m)}(W) \ = \Theta_{(\cZ^m)^{-1}\circ X(m)} (W^m) \ \text{ and } \ W = \left[ W , \Theta_{(\cZ^m)^{-1}\circ X(m)} (W^m) \right]_{\tilde T^{-1}(m)+1}.
	\end{equation}
\end{lemma}

\begin{proof}
	The fact that $W \mid_{\tilde T^{-1}(m)} \in \cF_{\leq m}$ is a direct consequence of the assumption $Z \in m \cF_{< T(1)}$ and the fact that $T(\bTm(m)) \leq m$ and is a stopping time.
	
	Consider now $m$ which is not a weak ascending ladder height time and let us prove the two relations of~\eqref{eq:discrete-snake}. Since $W = [W, \Theta_h(W)]_h$ for every $h \geq 0$, the second relation follows directly from the first one with $h = T^{-1}(m)$ since, when $m$ is not a weak ascending ladder height time, we have $T^{-1}(m) = \tilde T^{-1}(m) + 1$.
	
	As for the first one, $\Theta_{T^{-1}(m)}(W)$ is obtained by skipping the $T^{-1}(m)$ first excursions of $S$ reflected at its maximum. In terms of the shifted process $W^m$ at time $m$, this amounts to skipping the excursions needed to escape the ``valley'' in which $m$ sits, see Figure~\ref{fig:discrete-snake} for an explanation of this loose statement on a picture. The shifted process needs to reach level $X(m)$ to escape the valley, and the number of excursions needed to do so is precisely given by $(\cZ^m)^{-1} \circ X(m)$ by definition of $\cZ^m$. This proves the result.
\end{proof}

\subsection{Step 3: $\Pi_\infty \ci \cV_\infty$ by a perturbation argument} In this step we prove that $\Pi_\infty \ci \cV_\infty$ through a perturbation argument. We first introduce the perturbed sequence, explain the main idea and prove that $S_\infty \ci \cV_\infty$ in Section~\ref{subsub:perturbation}. The independence between $\Pi_\infty$ and $\cV_\infty$ is then established in Section~\ref{subsub:full-perturbation}.

\subsubsection{Perturbed sequence, main idea and $S_\infty \ci \cV_\infty$} \label{subsub:perturbation}

For each $p \geq 1$, let $(V'_n, \cP'_n)$ be a sequence of i.i.d.\ random variables, independent of $(V_n, \cP_n)$ and with the same law as that of $(V^*_p, \cP^*_p)$ conditioned on $V^*_p < \varepsilon / v_p$. Out of the sequence $(V'_n, \cP'_n)$ and the original sequence $(V_n,\cP_n)$, we define a new sequence of sticks
\[ \left( \widehat{V}_n, \widehat{\cP}_n \right) = \begin{cases}
	\left( V_n, \cP_n \right) & \text{ if } V_n < \varepsilon / v_p,\\
	\left( V'_n, \cP'_n \right) & \text{ else}.
\end{cases} \]
Then $(\widehat{V}_n, \widehat{\cP}_n)$ forms a sequence of i.i.d.\ random variables with common distribution the law of $(V'_1, \cP'_1)$ and we will denote with a hat \ $\widehat \cdot$ \ all the processes defined from this sequence of sticks, e.g., $\widehat \Pi$, $\widehat \Pi_p$, etc. Moreover, by construction we have
\[ \big( (\lambda_k, V_{\lambda_k}, \cP_{\lambda_k}), k \geq 1 \big) \ci \big( (\widehat V_n, \widehat \cP_n), n \geq 1 \big) \]
which implies for instance that
\[ \widehat \Pi_p \ci \cV^{\geq \varepsilon}_p \]
with $\cV^{\geq \varepsilon}_p = \cV_p - \cV^{< \varepsilon}_p$ obtained from $\cV_p$ by only keeping jumps $\geq \varepsilon$. If we can show that the perturbation induced by the $(V'_n, \cP'_n)$'s is negligible in the sense that $\Pi_p(t) - \widehat \Pi_p(t) \Rightarrow 0$ for each fixed $t$, then we will get $\Pi_\infty \ci \cV_\infty$ (because $\cV^{\geq \varepsilon}_p \Rightarrow_{p, \varepsilon} (\cV_\infty(t) - \texttt{d}t, t \geq 0)$). This is the object of the next section, and before going on we use these arguments to show that $S_\infty \ci \cV_\infty$.

The fact that $S_p - \widehat S_p \Rightarrow 0$ is a direct consequence of the assumption~\eqref{eq:independence-condition-2'} and the fact that $((V_n, \cP_n), n \geq 1)$ and $((\widehat V_n, \widehat \cP_n), n \geq 1)$ only differ by a locally finite number of terms (by Lemma \ref{lemma:subtrees-long-edges}). More precisely, we have
\begin{equation} \label{eq:bound-perturbation-S}
	\sup_{[0,t]} \left \lvert S_p - \widehat S_p \right \rvert \leq \frac{1}{p g_p} \sum_{k: \lambda_k \leq [pt]} \left( \left \lvert \cP_{\lambda_k} \right \rvert + \left \lvert \cP'_{\lambda_k} \right \rvert \right).
\end{equation}
By assumption we have $\lvert \cP_{\lambda_k} \rvert / (p g_p) \Rightarrow 0$. Since
\[ \P \left( \left \lvert \cP'_{\lambda_k} \right \rvert \geq x p g_p \right) = \P \left( \left \lvert \cP^*_p \right \rvert \geq x p g_p \mid V^*_p \leq \varepsilon v_p \right) \leq \frac{\P \left( \left \lvert \cP^*_p \right \rvert \geq x p g_p \right)}{\P \left( V^*_p \leq \varepsilon v_p \right)} \]
we also have $\lvert \cP'_{\lambda_k} \rvert / (p g_p) \Rightarrow 0$. Since the number of terms in the sum in the right-hand side of~\eqref{eq:bound-perturbation-S} forms a tight sequence, we get as desired that $S_p - \widehat S_p \Rightarrow 0$. We now show that the perturbation does not significantly change the chronological height process as well.

\subsubsection{Perturbation of the chronological spine process} \label{subsub:full-perturbation}

Let $((\widehat V^{(k)}_n, \widehat \cP^{(k)}_n), n \geq 1)$ be obtained by induction as $(\widehat V^{(0)}_n, \widehat \cP^{(0)}_n) = (V_n, \cP_n)$ and
\[ \left( \widehat{V}^{(k+1)}_n, \widehat{\cP}^{(k+1)}_n \right) = \begin{cases}
	\left( V'_n, \cP'_n \right) & \text{ if } n = \lambda_{k+1},\\
	\left( \widehat{V}^{(k)}_n, \widehat{\cP}^{(k)}_n \right) & \text{ else}.
\end{cases} \]
Thus, $((\widehat V^{(k)}_n, \widehat \cP^{(k)}_n), n \geq 1)$ and $((\widehat V^{(k+1)}_n, \widehat \cP^{(k+1)}_n), n \geq 1)$ only differ by one element and in order to prove that $\Pi_p(t) - \widehat \Pi_p(t) \Rightarrow 0$ it is enough to prove that $\widehat \Pi^{(k)}_p(t) - \widehat \Pi^{(k+1)}_p(t) \Rightarrow 0$ for every $k \geq 0$. Because 
 $\Pi_p(t) = (R_p \mid_{{\tilde T}_p^{-1}(t)}) \circ \vartheta^{[pt]}$, $(R_p, T_p) \Rightarrow (R_\infty, T_\infty)$ with $R_\infty$ a subordinator independent from $T_\infty$, it is enough to prove that
\[ \widehat R^{(k)}_p - \widehat R^{(k+1)}_p \Rightarrow 0 \ \text{ and } \ \widehat T^{(k)}_p - \widehat T^{(k+1)}_p \Rightarrow 0, \ k \geq 0. \]
By shifting and using the strong Markov property, it is enough to prove the result for $k = 0$. We show that $R_p - \widehat R^{(1)}_p \Rightarrow 0$, the convergence $T_p - \widehat T^{(1)}_p \Rightarrow 0$ can be shown along the same lines. To ease the notation let us define $\tilde \lambda = \frac{1}{p} \lambda_1$ and let us also denote with \ $\underline{\cdot}$ \ instead of with \ $\widehat{\cdot}^{(1)}$ \ all quantities defined from the $(\widehat V^{(1)}_n, \widehat \cP^{(1)}_n)$'s, e.g., $\underline{R}_p$ instead of $\widehat R^{(1)}_p$, etc.

For the sake of simplicity, let us now assume that the drift of the subordinator $\cZ_\infty$ is equal to $0$. We will briefly discuss in Remark~\ref{lem:patch} how to adapt our argument to the case of positive drift. Since $S_\infty \ci \cV_\infty$, $\lambda_1$ is with high probability not a record time of $S$ and so Lemma~\ref{prop:discrete-snake} entails (in an event of probability going to one)
\begin{equation}\label{eq:1}
	R = R \mid_{\bTm(\lambda_1)} + \Delta R \left( T^{-1}(\lambda_1) \right) \epsilon_{T^{-1}(\lambda_1)} + \Theta_{(\cZ^{\lambda_1})^{-1}(X(\lambda_1))} \left( R^{\lambda_1} \right) 
\end{equation}
and
\begin{equation}\label{eq:2}
	\underline R = \underline R \mid_{\underline{\widetilde T}^{-1}(\lambda_1)} + \Delta \underline R \left( \underline{T}^{-1}(\lambda_1) \right) \epsilon_{T^{-1}(\lambda_1)} + \Theta_{(\underline \cZ^{\lambda_1})^{-1}(\underline X(\lambda_1))} \left( \underline R^{\lambda_1} \right).
\end{equation}
In order to grasp more intuition on what follows, let us briefly give an interpretation of the previous relations. 
Recall that for the processes $R, T, \cZ$, time is measured in local time units for the reflected process $X$. The previous relation consists in decomposing the process $R$ (and $\bar R$) before and after the local time corresponding to the excursion of $X$ straddling $\lambda_1$. In particular, ${(Z^{\lambda_1})^{-1}(X(\lambda_1))}$ is the local time needed to exit the ``valley'' straddling $\lambda_1$.
See Figure \ref{fig:discrete-snake}.
 
Since the two initial sequences of sticks coincide up to $\lambda_1$, we have
\[ R \mid_{\widetilde T^{-1}(\lambda_1)} = \underline R \mid_{\underline{\widetilde T}^{-1}(\lambda_1)}. \]
Next, let $E$ be the event
\begin{multline*}
	\Big\{ \underbrace{\mbox{depth of the valley straddling $\lambda_1$} \ + \ \mbox{overshoot upon exiting the valley}}_{\mbox{for $S$}}\\
	\geq \ \underbrace{\mbox{depth of the valley}}_{\mbox{for $\bar S$}} \Big\}
\end{multline*}
which can be formerly defined as 
\[
E = \left\{ \cZ^{\lambda_1} \circ \left( \cZ^{\lambda_1} \right)^{-1} \circ X(\lambda_1) \geq \underline X(\lambda_1) \right\} \
\]
and analogously, define 
\[\ \underline E = \left\{ \underline \cZ^{\lambda_1} \circ \left( \underline \cZ^{\lambda_1} \right)^{-1} \circ \underline X(\lambda_1) \geq X(\lambda_1) \right\}.\]
In the event $E \cap \underline E$, the two Lukaziewicz paths exit the valley straddling $\lambda_1$ at the same time, i.e.,
\[ T \circ T^{-1}(\lambda_1) = \underline T \circ \underline T^{-1}(\lambda_1), \]
and this implies that, in this event,
\[ \Theta_{(Z^{\lambda_1})^{-1}(X(\lambda_1))} \left( R^{\lambda_1} \right) = \Theta_{(\underline Z^{\lambda_1})^{-1}(\underline X(\lambda_1))} \left( \underline R^{\lambda_1} \right). \]
We now claim that $\P(E), \P({\underline E})\to1$. Let us first consider the overshoot when exiting the valley straddling $\lambda_1$, which is given by
\[ \cZ^{\lambda_1} \circ \left( \cZ^{\lambda_1} \right)^{-1} \circ X(\lambda_1) - X(\lambda_1). \]
As $X(\lambda_1) \in m\cF_{<\lambda_1}$ and $\cZ^{\lambda_1}\in m\cF_{\geq\lambda_1}$, the latter quantity converges in distribution (after proper space rescaling by $1/p g_p$)
to $\cZ_\infty\circ \cZ_\infty^{-1}(\xi)-\xi$, where $\xi \ci \cZ_\infty$ is almost surely positive. In the absence of drift for the subordinator
$\cZ_\infty$, we must have
\[ \cZ_\infty\circ \cZ_\infty^{-1}(\xi)-\xi >0 \]
because $S_\infty$ does not creep at level $\xi$ (see for instance~\cite[Thereom VI.19]{Bertoin96:0}) and as a consequence, $\cZ^{\lambda_1} \circ \left( \cZ^{\lambda_1} \right)^{-1} \circ X(\lambda_1) - X(\lambda_1)$ 
is of the order of $1/(pg_p)$. 
On the other hand, $\underline X(\lambda_1) - X(\lambda_1)$ is negligible compared to $1/(pg_p)$ because of~\eqref{eq:independence-condition-2'} which implies that $\P(E) \to 1$. By similar arguments one can prove that $\P(\underline E) \to 1$. This shows that $\P(E), \P({\underline E})\to1$, and as discussed earlier, the last terms on the RHS of~\eqref{eq:1} and~\eqref{eq:2} are equal.

In order to conclude the proof, we only have to show that for the middle terms of~\eqref{eq:1} and~\eqref{eq:2} we have
\[ \P \left( c_p \Delta R(T^{-1}(\lambda_1)) \geq x \right) \mathop{\longrightarrow}_{p \to \infty} 0, \ \P \left( c_p \Delta \underline R(\underline T^{-1}(\lambda_1)) \geq x\right) \mathop{\longrightarrow}_{p \to \infty} 0 \]
for any $x > 0$. We have 
\[ \Delta R(T^{-1}(\lambda_1)) = \Delta R^{\lambda_1} \left( (\cZ^{\lambda_1})^{-1} (X(\lambda_1)) \right). \]
As $X(\lambda_1) \in m\cF_{<\lambda_1}$ and $(R^{\lambda_1}, \cZ^{\lambda_1}) \in m\cF_{\geq\lambda_1}$ we see that $c_p \Delta R(T^{-1}(\lambda_1))$ is equal in distribution to $\Delta R_p (\cZ^{-1}_p(\zeta_p))$ with $(R_p, \cZ_p) \ci \zeta_p$ and $\zeta_p$ equal in distribution to $X_p(\tilde \lambda)$. Thus $\zeta_p$ converges in distribution to some $\zeta_\infty$ independent from $\cZ_\infty$, so that $\cZ^{-1}_p(\zeta_p) \Rightarrow \cZ^{-1}_\infty(\zeta_\infty)$ and because $R_\infty \ci \cZ_\infty$ we obtain $\Delta R_p(\cZ^{-1}_p(\zeta_p)) \Rightarrow 0$. In the event $E \cap \underline E$ we have
\[ \Delta \underline R(\underline T^{-1}(\lambda_1)) = \Delta R^{\lambda_1} \left( (Z^{\lambda_1})^{-1} (\underline X(\lambda_1)) \right) \]
and so the same argument as above gives $c_p \Delta \underline R(\underline T^{-1}(\lambda_1)) \Rightarrow 0$. This shows that $\cV_\infty \ci \Pi_\infty$
which completes the proof of Theorem~\ref{thm:contour'}.

\begin{remark}\label{lem:patch}
	As mentioned in the proof, we only proved the previous result assuming that $\cZ_\infty$ has no drift. In this case, we argued that the two random walks exit the $\lambda_1$-valley at the same time with high probability. In the presence 
	of drift, this is not the case anymore. The two random walks $S$ or $\bar S$ can exit the valley by ``creeping''. In this case, the exit times are not equal, but 
	their differences vanish (macroscopically) at the limit. A similar but more cumbersome argument can then be applied, but the spirit of the proof remains the same.
\end{remark}

\section{Applications} \label{sec:applications}

We now come back to the two specific examples of Section~\ref{sub:examples}.

\subsection{Proof of Theorem~\ref{thm:finite-variance}}

In the non-triangular case and when $\lvert \cP^* \rvert$ has finite variance, Assumptions~\ref{ass-G},~\ref{ass-C1} and~\ref{ass-C2} imply Assumptions~\ref{ass-G'},~\ref{ass-C1'} and~\ref{ass-C2'} with $\psi(\lambda) = \lambda^2$, and so in order to prove Theorem~\ref{thm:finite-variance} we only have to check that conditions~\eqref{eq:independence-condition} and~\eqref{eq:independence-condition-2} hold. We first check~\eqref{eq:independence-condition-2}. Fix $\varepsilon > 0$ and let $\lambda = \inf\{n\geq0: V_n \geq \eps / v_p \}$. Then for every $K>0$
\[ \P\left( \lvert \cP^* \rvert \geq x \sqrt{p} \mid v_p V^* \geq \eps \right) = \P \left( \lvert \cP_{\lambda} \rvert \geq x \sqrt{p} \right) \leq \P \left( \lambda \geq [Kp] \right) + \P \left( \max_{i\in\{0,\ldots,[Kp]\}} \lvert \cP_i \rvert \geq x \sqrt{p} \right). \]
Since $\lvert \cP^* \rvert$ has finite variance, the second term vanishes when we let $p \to \infty$, while since $\cV_p \Rightarrow \cV_\infty$ we have that $\lambda/p$ is tight, and so the first term also vanishes when we let first $p \to \infty$ and then $K \to \infty$.

We now turn to~\eqref{eq:independence-condition}. We have
\[ \P \left( c_p R(1) \geq \varepsilon \mid T(1) \geq \delta p \right) = \frac{\P \left( c_p R(1) \geq \varepsilon , T(1) \geq \delta p \right)}{\P \left( T(1) \geq \delta p \right)}. \]

In the finite variance case, the tail of the ladder height time of a recurrent, zero-mean random walk with finite variance decays asymptotically like $x^{-1/2}$ up to a multiplicative constant. Applying this result to $S$ and $-S$, we obtain the existence of finite and positive constants $\kappa_1, \kappa_2$ such that
\[ \P(T(1) \geq x) \sim \kappa_1 x^{-1/2} \ \mbox{and } \ \P(\tau^-_1 \geq x) \sim \kappa_2 x^{-1/2} \ \text{ as } \ x \to \infty, \]
and in particular, in order to prove~\eqref{eq:independence-condition} we only have to prove that
\[ \sqrt p \ \P \left( c_p R(1) \geq \varepsilon , T(1) \geq \delta p \right) \to 0. \]
Since $x \mapsto 1 - e^{-x}$ is decreasing, Markov inequality gives
\[ \P \left( c_p R(1) \geq \varepsilon , T(1) \geq \delta p \right) \leq \frac{1}{(1 - e^{-\varepsilon})(1-e^{-\delta})} \E \left( \left( 1 - e^{-c_p R(1)} \right) \left( 1 - e^{-T(1)/p} \right) \right) \]
and so in order to prove the result, it is enough to show that
\[ \sqrt{p} \E \left( \left( 1 - e^{-c_p R(1)} \right) \left( 1 - e^{-T(1)/p} \right) \right) \mathop{\longrightarrow}_{p \to \infty} 0. \]
Considering $F(t,z,r) = (1 - e^{- t / p}) (1 - e^{-c_p r})$ in~\eqref{eq:F} and using the notation $\xi = \lvert \cP^* \rvert$, we obtain
\begin{align*}
	\E \left( \left(1-e^{- T(1) / p} \right) \left( 1 - e^{- c_p R(1)} \right) \right) & = \sum_{t,x,z} \E \left( \left(1 - e^{- t / p} \right) \left( 1 - e^{- c_p A_z(\cP^*)} \right) ; \xi = x+z \right) \P(\tau^-_{x-1} = t-1)\\
	& = \sum_{x,z} \E \left( 1 - e^{- c_p A_z(\cP^*)} ; \xi = x+z \right) \E\left( 1 - e^{-(\tau^-_{x-1}+1) / p} \right)\\
	& = \sum_{x \geq 1} \E \left( 1 - e^{- c_p A_{\xi-x}(\cP^*)} ; \xi \geq x \right) \left( 1 - e^{-1/p} u^x_p \right)
\end{align*}
where the initial sum is taken over $t,x \geq 1$ and $z \geq 0$ and $u_p = \E(e^{-\tau^-_1 / p})$. We thus have
\[ \sqrt{p} \E \left( \left(1 - e^{- T(1) / p} \right) \left( 1 - e^{-c_p R(1)} \right) \right) = \sqrt{p}\sum_{x \geq 1} \E \left( 1 - e^{-c_p A_{\xi - x}(\cP^*)} ; \xi \geq x \right) \left( 1 - u^x_p \right) + o(1) \]
with $o(1)$ an error term that vanishes as $p \to \infty$. We split the sum into two terms, depending on whether $x$ is large or not. Fix until further notice some $\eta > 0$, and consider the terms $x \leq \eta \sqrt{p}$: we have
\begin{align*}
	\sqrt{p}\sum_{x \leq \eta \sqrt{p}} \E \left( 1 - e^{-c_p A_{\xi-x}(\cP^*)} ; \xi \geq x+1 \right) & \left( 1 - u^x_p \right)\\
	& \hspace{-10mm} \leq \sqrt{p}\left( 1 - u^{\eta \sqrt{p}}_p \right) \sum_{x \leq \eta \sqrt{p}} \E \left( 1 - e^{-c_p A_{\xi-x}(\cP^*)} ; \xi \geq x+1 \right)\\
	& \hspace{-10mm} \leq \sqrt{p}\left( 1 - u^{\eta \sqrt{p}}_p \right) \sum_x \E \left( 1 - e^{-c_p A_{\xi-x}(\cP^*)} ; \xi \geq x+1 \right)\\
	& \hspace{-10mm} = \sqrt{p}\E \left( 1 - e^{-c_p R(1)} \right) \times \left( 1 - u^{\eta \sqrt{p}}_p \right) = 1 - u^{\eta \sqrt{p}}_p.
\end{align*}
From the tail behavior $\P(\tau^-_1 \geq x) \sim \kappa_2 x^{-1/2}$, standard Tauberian theorems imply the existence of a finite constant $\kappa$ such that $\E(1 - e^{-\lambda \tau^-_1}) \sim \kappa \lambda^{1/2}$ as $\lambda \to 0$ (see for instance \cite[Corollary~$8.1.7$]{Bingham89:0}), which implies in particular that $1 - u^{\eta \sqrt p}_p \to 1 - e^{-\eta \kappa}$. Thus we obtain
\[ \limsup_{p \to \infty} \sqrt{p}\sum_{x \leq \eta \sqrt{p}} \E \left( 1 - e^{-c_p A_{\xi-x}(\cP^*)} ; \xi \geq x+1 \right) \left( 1 - u^x_p \right) \leq 1 - e^{-\eta \kappa}. \]
Let us now look at the other terms corresponding to $x \geq \eta \sqrt{p}$: we have
\begin{align*}
	\sqrt{p}\sum_{x \geq \eta \sqrt{p}} \E \left( 1 - e^{-c_p A_{\xi-x}(\cP^*)} ; \xi \geq x+1 \right) & \left( 1 - u^x_p \right)\\
	& \hspace{-25mm} \leq \sqrt{p}\sum_{x \geq \eta \sqrt{p}} \E \left( 1 - e^{-c_p A_{\xi-x}(\cP^*)} ; \xi \geq x+1 \right)\\
	& \hspace{-25mm} = \sqrt{p}\sum_{x \geq \eta \sqrt{p}} \E \left( 1 - e^{-c_p A_{\xi-x}(\cP^*)} ; \xi \geq x+1, \xi \geq \eta \sqrt{p}\right)\\
	& \hspace{-25mm} \leq \sqrt{p}\sum_x \E \left( 1 - e^{-c_p A_{\xi-x}(\cP^*)} ; \xi \geq x+1, \xi \geq \eta \sqrt{p}\right).
\end{align*}
From~\eqref{eq:F-0} we obtain
\[ \E \left( 1 - e^{-c_p R(1)}; \lvert \cP_{T(1) - 1} \rvert \geq \eta \sqrt{p}\right) = \sum_x \E \left( 1 - e^{-c_p A_{\xi-x}(\cP^*)} ; \xi \geq x+1, \xi \geq \eta \sqrt{p}\right) \]
and so using Cauchy--Schwarz inequality we get
\begin{align*}
	\sqrt{p}\sum_{x \geq \eta \sqrt{p}} \E \left( 1 - e^{-c_p A_{\xi-x}(\cP^*)} ; \xi \geq x+1 \right) & \left( 1 - u^x_p \right)\\
	& \hspace{-25mm} \leq \sqrt{p} \sqrt{ \E \left[ \left( 1 - e^{-c_p R(1)} \right)^2 \right] \P \left( \lvert \cP_{T(1) - 1} \rvert \geq \eta \sqrt{p}\right) }\\
	& \hspace{-25mm} \leq \eta^{-1/2} \sqrt{ \sqrt{p}\E \left( 1 - e^{-2c_p R(1)} \right) \E \left( \left \lvert \cP_{T(1)-1} \right \rvert ; \left \lvert \cP_{T(1)-1} \right \rvert \geq \eta \sqrt{p}\right) }.
\end{align*}
Since $\lvert \cP^* \rvert$ is assumed to have finite variance, $\lvert \cP_{T(1)-1} \rvert$ has finite mean and so as $p \to \infty$ we have $\E(\lvert \cP_{T(1)-1} \rvert ; \lvert \cP_{T(1)-1} \rvert \geq \eta \sqrt{p}) \to 0$. On the other hand, the fact that $R(1)$ is in the domain of attraction of a $\beta$-stable distribution and the choice of $c_p$ (which ensures that $R_\infty$ has Laplace exponent $\lambda^\beta$) implies that $\sqrt{p}\E(1 - e^{-2c_p R(1)}) \to 2^\beta$, and so we get
\[ \sqrt{p}\sum_{x \geq \eta \sqrt{p}} \E \left( 1 - e^{-c_p A_{\xi-x}(\cP^*)} ; \xi \geq x+1 \right) \left( 1 - u^x_p \right) \mathop{\longrightarrow}_{p \to \infty} 0. \]
We have thus proved that for any $\eta > 0$,
\[ \limsup_{p \to \infty} \sqrt{p}  \E \left( \left( 1- e^{- t T(1) / p} \right) \left( 1 - e^{-c_p R(1)} \right) \right) \leq 1 - e^{-\eta \kappa}. \]
Letting $\eta \to 0$ concludes the proof of~\eqref{eq:independence-condition}.

\subsection{Proof of a triangular version of Theorem~\ref{thm:extended-BH}}

The following result extends Theorem~\ref{thm:extended-BH} to a triangular setting.

\begin{theorembis} {\ref{thm:extended-BH}'} [Triangular version of Theorem~\ref{thm:extended-BH}] \label{thm:extended-BH'}
	Assume that:
	\begin{itemize}
		\item Assumptions~\ref{ass-G'} and~\ref{ass-C2'} are satisfied and $V^*_p$ and $\lvert \cP^*_p \rvert$ are independent;
		\item conditionally on $(V^*_p, \lvert \cP^*_p \rvert) = (v,n)$, the locations of the atoms of $\cP^*_p$ are i.i.d.\ with common distribution $v X_p$ for some random variable $X_p \in (0,1]$;
		\item $\liminf_p \E(X_p) > 0$.
	\end{itemize}
	Then Assumption~\ref{ass-C1'} is satisfied with $c_p = v_{[1/g_p]}/\E(X_p)$ and $R_\infty = \cV_\infty$, and~\eqref{eq:independence-condition'} and~\eqref{eq:independence-condition-2'} hold. In particular, $\Pi_\infty$ is the $\psi/\varphi$-snake with $\varphi$ the Laplace exponent of $\cV_\infty$.
\end{theorembis}

\begin{proof}
	We need to check Assumption~\ref{ass-C1'} and that the asymptotic independence conditions~\eqref{eq:independence-condition'} and~\eqref{eq:independence-condition-2'} hold. Let for simplicity define $\xi = \lvert \cP_{T(1)-1} \rvert$. It is well-known that conditionally on $\xi = n$, $Z(1)$ is uniformly distributed on $\{0, \ldots, n-1\}$ (this can be checked from~\eqref{eq:F-0}). Given the law of $(V^*_p, \cP^*_p)$, it follows that $R(1)$ conditioned on $\xi = n$ is equal in distribution to $V^*_p X_p^{(U:n)}$ where $V^*_p, X^1_p, \ldots, X^n_p$ and $U$ are independent, the $X^k_p$ are i.i.d.\ distributed as $X_p$, $U$ is uniformly distributed on $\{1, \ldots, n\}$ and $X_p^{(k:n)}$ is the $k$th order statistic of the $(X^1_p, \ldots, X^n_p)$. Since $X_p^{(U:n)}$ is equal in distribution to $X_p$ by exchangeability, we obtain that $R(1)$ is independent from $\xi$, and thus from $T(1)$, and is equal in distribution to $V^*_p X_p$ with $V^*_p \ci X_p$. 

	It follows that Assumption~\ref{ass-C1'} is satisfied with $c_p = v_{[1/g_p]} / \E(X_p)$ and $R_\infty = \cV_\infty$. The independence condition~\eqref{eq:independence-condition} also holds since we have just argued that 
	\[\P_p(c_p R(1) \geq \varepsilon \mid T(1) \geq \delta p) = \P(c_p V^*_p X_p \geq \varepsilon) \to 0. \] 
	Further, (\ref{eq:independence-condition-2})
	holds since $V^*_p \ci \lvert \cP^*_p \rvert$ by assumption and $p g_p \to \infty$.
\end{proof}

\section{Proof of Proposition \ref{prop:more-ind}}\label{sect:preliminaries}

In the following, for $m \in \N$, $t \geq 0$ and $p\in\N \cup \{\infty\}$ we define 
\begin{equation} \label{eq:def-X}
	X(m) = \sup_{\{0,\ldots,m\}} S - S(m) \ \text{ and } \ X_p(t) = \sup_{[0,t]} S_p - S_p(t) = - \inf_{[0,t]} \hat S^t_p.
\end{equation}

To prove Proposition \ref{prop:more-ind} we will repeatedly use the following two simple lemmas. The first lemma is obvious but the situation it considers will actually be often encountered in what follows. The second lemma follows from Lemma~\ref{lemma:continuity-properties} and the assumption $(R_p, T_p) \Rightarrow (R_\infty, T_\infty)$ with $R_\infty \ci T_\infty$ (so that that $R_\infty$ is almost surely continuous at any random time $\Gamma \in m \sigma(T_\infty)$). Recall also that under Condition~\ref{ass-G'} we have proved in Proposition~\ref{prop:convergence-genealogy} that $T^{-1}_p \Rightarrow T^{-1}_\infty$ which is also needed in the proof.

\begin{lemma}\label{lemma:ind}
	If the pair $(A_\infty, B_\infty)$ is the weak limit of a sequence $(A_p, B_p)$ such that $A_p \in m \cF_{< [pt]}$ and $B_p \in m \cF_{\geq [pt]}$ for some $t \geq 0$, then $A_\infty \ci B_\infty$.
\end{lemma}

\begin{lemma}\label{lemma:conv-RmidT}
	If Condition~\ref{ass-G'} holds and $(R_p, T_p) \Rightarrow (R_\infty, T_\infty)$ with $R_\infty \ci T_\infty$, then for any $t \geq 0$ we have $R_p\mid_{T^{-1}_p(t)} \Rightarrow R_\infty\mid_{T^{-1}_\infty(t)}$ as well as $R_p\mid_{T^{-1}_p(t)-} \Rightarrow R_\infty\mid_{T^{-1}_\infty(t)}$.
\end{lemma}

For $0 \leq s \leq t$ let
\[ \Xi^s(t) = \left( T^s_\infty \mid_{(T^s_\infty)^{-1}(t-s)-}, \cZ^s_\infty \mid_{(T^s_\infty)^{-1}(t-s)-} \right). \]

Since $R^u_p$ is distributed as $R_p$, the sequence $(R^u_p)$ is tight. In the following, we will assume without loss of generality by working along appropriate subsequences that $R^u_p \Rightarrow R^u_\infty$ with $R^u_\infty$ equal in distribution to $R_\infty$, and that this convergence holds jointly with any other random variables needed. In particular, the $R^u_\infty$ are assumed to live on the same probability space as all the other random variables previously defined, in particular $S_\infty$ and $R_\infty$.

\subsection{Subordinator decomposition}

We first extend the snake-like property of Proposition~\ref{prop:discrete-snake} to the continuum. For further notice we make the following remark.

\begin{remark} \label{rk:proba-1}
	Under Condition~\ref{ass-G'}, $S$ suitably rescaled converges to a L\'evy process with infinite variation and so the condition ``$m$ is not a weak record time'' holds with high probability. In particular, without loss of generality and in order to avoid cumbersome statements we will henceforth assume that~\eqref{eq:discrete-snake} always holds, thereby neglecting an event of vanishing probability.
\end{remark}

The main application of Proposition~\ref{prop:discrete-snake} is to disentangle the dependence of the subordinators $R_\infty, T_\infty, \cZ_\infty$ in the past before time $t$ and in the future after time $t$, where time is now measured in real time units. This is the purpose of the next result.

\begin{prop} \label{prop:T-Z}
	For any $s \leq t \in \R$ such that $\Delta S_\infty(t-s) = 0$, $(T^s_\infty, \cZ^s_\infty)$ can be expressed as a measurable function of $T^t_\infty$, $\cZ^t_\infty$, $\Xi^s(t)$ and $X_\infty(t)$.
\end{prop}

\begin{proof}
	Upon shifting time, it is enough to prove the result for $s = 0$. By definition we have
	\begin{equation} \label{eq:expression-1}
		T_\infty = \left[ T_\infty, \Theta_{T^{-1}_\infty(t)}(T_\infty) \right]_{T_\infty^{-1}(t)-} + \epsilon_{T_\infty^{-1}(t)} \Delta T_\infty \left( T^{-1}_\infty(t) \right) 
	\end{equation}
	and
	\begin{equation} \label{eq:expression-2}
		\cZ_\infty = \left[ \cZ_\infty, \Theta_{T^{-1}_\infty(t)}(\cZ_\infty) \right]_{T_\infty^{-1}(t)-} + \epsilon_{T_\infty^{-1}(t)} \Delta \cZ_\infty \left( T^{-1}_\infty(t) \right).
	\end{equation}
		
	Now we claim that the shifted processes are given by
	\begin{equation} \label{eq:shifted}
		\Theta_{T^{-1}_\infty(t)}(T_\infty) = \Theta_{(\cZ^t_\infty)^{-1}\circ X_\infty(t)}(T^t_\infty) \ \text{ and } \ \Theta_{T^{-1}_\infty(t)}(\cZ_\infty) = \Theta_{(\cZ^t_\infty)^{-1}\circ X_\infty(t)}(\cZ^t_\infty)
	\end{equation}
	while the jumps are given by
	\begin{equation} \label{eq:jumps}
		\Delta T_\infty(T^{-1}_\infty(t)) = T^t_\infty \circ (\cZ^t_\infty)^{-1}\circ X_\infty(t) + \left(t-T_\infty(T^{-1}_\infty(t)-)\right)
	\end{equation}
	and
	\begin{equation} \label{eq:expression-last}
		\Delta \cZ_\infty(T^{-1}_\infty(t)) = \Delta \cZ^t_\infty \left( T^t_\infty \circ (\cZ^t_\infty)^{-1} \circ X_\infty(t) \right).
	\end{equation}
	
	Indeed, all relations~\eqref{eq:shifted}--\eqref{eq:expression-last} hold at the discrete level, i.e., by replacing $\infty$ by $p < \infty$: for~\eqref{eq:shifted} this is obtained after scaling from Lemma~\ref{prop:discrete-snake}; for~\eqref{eq:jumps} and~\eqref{eq:expression-last} this comes from the fact that $T(T^{-1}(m)) - m = T^m \circ (\cZ^m)^{-1} \circ X(m)$ which expresses the fact that, for the shifted process, the time needed to exit the ``valley'' straddling $t$ is equal to the time needed to go above level $X(m)$ (see Figure \ref{fig:discrete-snake}): this directly implies~\eqref{eq:jumps} and also~\eqref{eq:expression-last} because both sides then correspond to the overshoot when exiting the valley.
	
	Since these relations hold at the discrete level, $(T_p, \cZ_p) \Rightarrow (T_\infty, \cZ_\infty)$ and we consider a continuity point of all the processes involved, we can invoke Lemma~\ref{lemma:continuity-properties} to justify the passage to the limit $p = \infty$ and thus obtain~\eqref{eq:shifted}--\eqref{eq:expression-last}. This completes the proof of the result.
\end{proof}

A direct consequence of the next result is that Proposition~\ref{prop:more-ind} holds for negative indices. Recall the random variables $R^u_\infty$ introduced after Lemma~\ref{lemma:conv-RmidT}, which are weak limits of the $R^u_p$ assumed to live on the same probability space than $S_\infty$.

\begin{corollary} \label{cor:ind}
	Assume that the assumptions of Theorem~\ref{thm:height'} hold. Then for any finite set $J\subset \R \cap (-\infty, 0)$ we have
	\[ R_\infty \ci \left(\left( \cZ_\infty^u, T^u_\infty, R^u_\infty\mid_{(T^u_\infty)^{-1}(-u)} \right)_{u \in J}, \hat S^0_\infty \right). \]
\end{corollary}

\begin{proof}
	Let $G_\infty = ( R^u_\infty\mid_{(T^u_\infty)^{-1}(-u)} )_{u\in J}$ and $G_p = ( R^u_p\mid_{(T^u_p)^{-1}(-u)-} )_{u\in J}$ for $p \in \N$. By working under appropriate subsequences, it follows from Lemma~\ref{lemma:conv-RmidT} that
	\[ \left( G_p, \hat S_p^0, R_p, \cZ_p, T_p \right) \Rightarrow \left( G_\infty, \hat S_\infty^0, R_\infty, \cZ_\infty, T_\infty \right) \]
	Since for finite $p \in \N$ we have $(G_p,\hat S_p^0) \in m \cF_{<0}$ while $(R_p, \cZ_p, T_p) \in m \cF_{\geq 0}$, Lemma~\ref{lemma:ind} shows that $(G_\infty, \hat S_\infty^0) \ci \left(R_\infty, \cZ_\infty, T_\infty \right)$. Since by assumption we have $R_\infty \ci \left(\cZ_\infty, T_\infty\right)$, this gives further $R_\infty \ci (\cZ_\infty, T_\infty, G_\infty, \hat S^0_\infty)$. This finally entails the desired result as the remaining random variables $(\cZ^u_\infty, T^u_\infty)_{u\leq 0}$ are measurable with respect to $(\cZ_\infty, T_\infty, \hat S^0_\infty)$ in view of Proposition~\ref{prop:T-Z} (because $T_\infty \mid_{T^{-1}_\infty(t)-}$, $\cZ_\infty \mid_{T^{-1}_\infty(t)-}$ and $X_\infty(t)$ are measurable with respect to $\hat S^t_\infty$).
\end{proof}

\begin{remark}
	Contrary to the shifted process $T_\infty^s$ that can be directly defined through the shifted path $S_\infty^s$, the process $R_\infty^s$
	has no other obvious definition at the continuum than its existence in terms of the limit of its discrete counterpart.
\end{remark}	

We will finally need the following result, whose proof uses the same continuity arguments as in the proof of Proposition~\ref{prop:T-Z}. Note in particular (and this is rather crucial) that $R_\infty$ and $R^s_\infty$ are continuous at the points considered because $R_\infty \ci T_\infty$ and $R^s_\infty \ci (\cZ^s_\infty, X_\infty(s))$. Note also that in contrast with the similar decompositions~\eqref{eq:expression-1} and~\eqref{eq:expression-2} for $T_\infty$ and $\cZ_\infty$, there is no extra atom in the decomposition of $R_\infty$ because of this independence structure.

\begin{prop}\label{lem:constinuous-snake-L}
	Assume that the assumptions of Theorem~\ref{thm:height'} hold. Then for every $s \geq 0$, we have almost surely that
	\begin{equation}\label{eq:continuous-snake-L}
		 R_\infty = \left[ R_\infty, \Theta_{(\cZ^s_\infty)^{-1}\circ X_\infty(s)}(R^s_\infty) \right]_{T^{-1}_\infty(s)} = \left[ R_\infty, \Theta_{(\cZ^s_\infty)^{-1}\circ X_\infty(s)}(R^s_\infty) \right]_{T^{-1}_\infty(s)-}.
	\end{equation}
\end{prop}

\subsection{Another path decomposition}
The proof of Proposition~\ref{prop:more-ind}
relies on one extra path decomposition presented now.
We begin with the following result, which is a consequence of excursion theory applied to the process reflected at its supremum.

\begin{lemma}
	Let $(\cG_t)$ be a filtration such that $S_\infty$ is still a L\'evy process with respect to this filtration. Fix some $\tau\geq0$ and let $e := \left(S_\infty(t+H) - S_\infty(H), t\in[0,\tau - H]\right)$ with $H= T_\infty (T^{-1}_\infty(\tau)-)$ be the last negative excursion of $S_\infty$ away from its supremum. If $X \in m\cG_H$ then $X$ and $e$ are independent conditionally on $H$.
\end{lemma}

\begin{corollary}\label{lemma:subtelty}
	Let $J \subset \R$ be any finite set with $\min J = 0$, $\tau > \max J$ and $e$ and $H$ defined as in the previous lemma. 
	Then $e$ and $R_\infty \mid_{T_\infty^{-1}(\tau)}$ are independent conditionally on the shifted/stopped processes $(\Xi^s(\tau), s \in J)$.
\end{corollary}
\begin{proof}
	The main idea is to apply the previous lemma with the filtration $\cG_t = \sigma(Z_t)$ where
	\[ Z_t = \left( R_\infty \mid_{T_\infty^{-1}(t)}, S_\infty \mid_t \right), \ t \geq 0, \]
	 and $X = Z_H$. We decompose the proof into two steps.
	
	\medskip
	
\noindent \textit{Step 1.} We first prove that the assumptions of the previous result hold: by definition we have $X \in m \cG_H$ and so we need to show that $S_\infty$ is a L\'evy process in the filtration $(\cG_t)$. To do so we only have to prove that for any bounded and continuous functions $f$ and $g$ and any $t \geq 0$ we have
	\[ \E \left[ f \left( \Theta_t(S_\infty) \right) g \left( R_\infty\mid_{T_\infty^{-1}(t)}, S_\infty \mid_t \right) \right] = \E \left[ f \left(S_\infty \right) \right] \E \left[ g \left( R_\infty\mid_{T^{-1}_\infty(t)}, S_\infty \mid_t \right) \right]. \]
	For finite $p$ this is true as
	\[ \Theta_t(S_p) \in m \cF_{\geq [pt]} \ \text{ while } \ \left( R_p\mid_{T_p^{-1}(t)-}, S_p \mid_t \right) \in m \cF_{< [p t]} \]
	and so the result follows by letting $p \to \infty$.
	\\

	\noindent \textit{Step 2.} In the first step we have proved that the assumptions of the previous lemma hold, which gives the independence between $Z_H$ and the excursion $e$ conditionally on $H$. By definition of $H$, $S_\infty$ does not accumulate any local time at its maximum on $[H, \tau]$ and so we have $T_\infty^{-1}(\tau) = T_\infty^{-1}(H)$. Thus, we obtain the independence between
	\begin{equation} \label{eq:ind}
		\left( R_\infty\mid_{T_\infty^{-1}(\tau) }, S_\infty \mid_H \right) \ \text{ and } \ e
	\end{equation}
	conditionally on $H$. Let $\sigma^* = \min\{s \in J \cup \{\tau\}: s \geq H\}$ and define	
	\[ A = \left(\Xi^s(\tau), s \in J \cap [0,\sigma^*) \right) \ \text{ and } \ B = \left( \Xi^s(\tau), s \in J \cap [\sigma^*, \tau] \right). \]
	Then by definition of $H$ we have $A\in m \sigma \left( S_\infty \mid_H \right)$ while $B \in m \sigma \left( e \right)$. Thus~\eqref{eq:ind} implies that
	\[ \left( R_\infty\mid_{T_\infty^{-1}(\tau)}, A\right) \ \text{ and } \ \left(e,B\right) \]
	are independent conditionally on $H$. If $X \ci (Y, Z)$, then $X$ and $Y$ are independent conditionally on $Z$: using this observation twice, we obtain from the previous independence statement that
	\[ \left( R_\infty\mid_{T_\infty^{-1}(\tau)} \right) \ \text{ and } \ e \]
	are independent conditionally on $(H, A, B)$ and so also simply conditionally on $(A, B)$ since $H \in m \sigma(A)$. This completes the proof of the result.
\end{proof}

\subsection{Proof of Proposition~\ref{prop:more-ind}}
We decompose the proof into several steps. The last step of the proof will use the following lemma.

\begin{lemma}\label{lem:sub}
	If $S_1$ and $S_2$ are two i.i.d.\ subordinators, then for any random time $\kappa$ independent from $(S_1, S_2)$, $[S_1, S_2]_\kappa$ is distributed as $S_1$ and is independent of~$\kappa$.
\end{lemma}

\noindent \textit{Step 0.} The very first step is to show that it is enough to prove the result for finite sets $J$ such that $0 = \min J$. Indeed, assume this is the case and consider any finite set $J \subset \R$ with $0 \in J$. Write $J = J_- \cup J_+$ where $J_-$ consists of all the strictly negative indices and $J_+ \subset \R_+$. Lemma~\ref{lemma:ind} implies that
\[ \left(R_\infty,\left( \cZ^s_\infty, T^s_\infty \right)_{s \in J_+} \right)\ci \hat S^0_\infty \]
and thus
\[ R_\infty \ci \left(\hat S^0_\infty,\left( \cZ^s_\infty, T^s_\infty \right)_{s \in J_+} \right) \]
since we assume $R_\infty \ci \left( \cZ^s_\infty, T^s_\infty \right)_{s \in J_+}$. Since $\min J_+ = 0$ and $(\cZ^u_\infty ,T^u_\infty )\in m \sigma(\cZ_\infty, T_\infty, \hat S^0_\infty)$ for any $u \leq 0$ by Proposition~\ref{prop:T-Z} we get the desired result.
\\

Therefore, the rest of the proof is devoted to proving that
\[ R_\infty \ci \left( T^t_\infty, \cZ^t_\infty \right)_{t \in J} \]
for any finite set $J \subset \Q$ with $\min J = 0$. The proof operates by induction on $\lvert J \rvert \geq 1$. For $\lvert J \rvert = 1$ this is simply 
$R_\infty \ci \left(\cZ_\infty, T_\infty\right)$ which holds by assumption, so let $J \subset \Q$ be a finite set with $\min J = 0$, let $\tau > \max J$ and $J_+ = J \cup \{\tau\}$: assuming that $R_\infty \ci (T^s_\infty, \cZ^s_\infty)_{s \in J}$, the rest of the proof is devoted to proving that $R_\infty \ci (\cZ^s_\infty, T^s_\infty)_{s \in J_+}$, i.e., that
\begin{equation} \label{eq:goal-ind}
	\E \left[ f \left( R_\infty \right) g \left( (\cZ_\infty^s, T^s_\infty)_{s \in J_+} \right) \right] = \E \left[ f \left( R_\infty \right) \right] \E \left[ g \left( (\cZ_\infty^s, T^s_\infty)_{s \in J_+} \right) \right]
\end{equation}
for any bounded measurable functions $f, g$. Recall the random variables $R_\infty^s$ introduced after Lemma~\ref{lemma:conv-RmidT} and in order to ease the notation, for $s \in J$ define 
\[ \Gamma = \Theta_{(\cZ^{\tau}_\infty)^{-1} \circ X_\infty(\tau)}(R_\infty^\tau) \]
so that 
\[R_\infty = \left[ R_\infty, \Gamma \right]_{T_\infty^{-1}(\tau)} \]
in view of Proposition~\ref{lem:constinuous-snake-L}. We now investigate in more details 
 the structure of this decomposition.

\bigskip

\noindent \textit{Step 1.}
A consequence of Corollary~\ref{cor:ind} (replacing $R_\infty$ with $R^\tau_\infty$ and taking $u=-\tau$) is that
\[ R^{\tau}_\infty \ci \left( (\cZ^s_\infty, T^s_\infty)_{s \in J_+}, R_\infty\mid_{T_\infty^{-1}(\tau)}, X_\infty(\tau) \right). \]
As $R^{\tau}_\infty$ is a $\beta$-stable subordinator, this entails that $\Gamma$
is a $\beta$-stable subordinator independent from $\left((\cZ^s_\infty, T^s_\infty)_{s \in J_+}, R_\infty\mid_{T_\infty^{-1}(\tau)}\right)$.
\\

\noindent \textit{Step 2.} In this step we prove that
\begin{equation} \label{eq:step-3}
	\E \left[ f\left( R_\infty\mid_{T_\infty^{-1}(\tau)} \right) \mid (\cZ_\infty^s, T^s_\infty)_{s \in J_+} \right] = 
	\E \left[ f\left( R_\infty\mid_{T_\infty^{-1}(\tau)} \right) \mid T_\infty^{-1}(\tau)\right].
\end{equation}
To prove this we will use the following independence and measurability results:
\begin{equation} \label{eq:step-2}
	\left( R_\infty\mid_{T_\infty^{-1}(\tau)}, (\Xi^s(\tau), X_\infty^s(\tau-s))_{s \in J} \right) \ci \left(\cZ^{\tau}_\infty, T^{\tau}_\infty\right)
\end{equation}
as a consequence of Lemma~\ref{lemma:ind}, and
\begin{equation}\label{eq:step-22}
\left(X_\infty^s(\tau-s), s\in J_+\right) \ \in m \sigma(e)
\end{equation}
where $e$ is defined as in Corollary~\ref{lemma:subtelty}. Indeed, either $e$ does not straddle $s$ and then $X_\infty^s(\tau-s) = -\min e$, or $e$ straddles $s$ and then $X_\infty^s(\tau-s)$ can be computed from~$e$.

Let us now proceed with the proof of~\eqref{eq:step-3}. By shifting at time $s$, Proposition~\ref{prop:T-Z} implies that for any $s \in J$ we have
\[ (\cZ_\infty^s, T^s_\infty) \in m \sigma \left( \cZ^{\tau}_\infty, T^{\tau}_\infty, \Xi^s(\tau), X_\infty^s(\tau-s) \right) \]
and thus
\[ (\cZ_\infty^s, T^s_\infty)_{s \in J_+} \in m \sigma \left( \cZ^{\tau}_\infty, T^{\tau}_\infty, \left(\Xi^s(\tau), X_\infty^s(\tau-s)\right)_{s\in J} \right). \]
In particular, the law of total expectation gives
\begin{multline}
	\E \left[ f\left( R_\infty\mid_{T_\infty^{-1}(\tau)} \right) \mid (\cZ_\infty^s, T^s_\infty)_{s \in J_+} \right] \\
	= \E\left(\E \left[ f\left( R_\infty\mid_{T_\infty^{-1}(\tau)} \right) \mid \cZ^{\tau}_\infty, T^{\tau}_\infty, \left(\Xi^s(\tau), X_\infty^s(\tau-s)\right)_{s\in J} \right] \mid (\cZ_\infty^s, T^s_\infty)_{s \in J_+} \right). \label{eq:hh}
\end{multline}
Next, if $X \ci (Y, Z)$, then $X$ and $Y$ are independent conditionally on $Z$: using this observation with the independence relation~\eqref{eq:step-2} allows to get rid of 
$\left(\cZ^\tau_\infty, T^{\tau}_\infty\right)$ in the conditioning, i.e.,
\begin{multline*}
	\E \left[ f \left( R_\infty\mid_{T_\infty^{-1}(\tau)} \right) \mid \cZ^\tau_\infty, T^{\tau}_\infty, (\Xi^s(\tau), X_\infty^s(\tau-s))_{s \in J} \right]\\
	= \E \left[ f \left( R_\infty\mid_{T_\infty^{-1}(\tau)} \right) \mid (\Xi^s(\tau), X_\infty^s(\tau-s))_{s \in J}\right]
\end{multline*}
which leads further to
\[ \E \left[ f\left( R_\infty\mid_{T_\infty^{-1}(\tau)} \right) \mid \cZ^\tau_\infty, T^{\tau}_\infty, (\Xi^s(\tau), X_\infty^s(\tau-s))_{s \in J}\right] = \E \left[ f\left( R_\infty\mid_{T_\infty^{-1}(\tau)} \right) \mid (\Xi^s(\tau))_{s \in J}\right] \]
according to Corollary~\ref{lemma:subtelty} and~\eqref{eq:step-22}. At this point, we have therefore proved that
\[ \E \left[ f\left( R_\infty\mid_{T_\infty^{-1}(\tau)} \right) \mid \cZ^\tau_\infty, T^{\tau}_\infty, (\Xi^s(\tau), X_\infty^s(\tau-s))_{s \in J} \right] 	= \E \left[ f\left( R_\infty\mid_{T_\infty^{-1}(\tau)} \right) \mid (\Xi^s(\tau))_{s \in J} \right] \]
and since 
\begin{equation}\label{eq:ddf}
(\Xi^s(\tau))_{s \in J} \in m \sigma(\cZ^s_\infty, T^s_\infty, s \in J) \subset m \sigma(\cZ^s_\infty, T^s_\infty, s \in J_+),
\end{equation}\eqref{eq:hh} implies that
\[ \E \left[ f\left( R_\infty\mid_{T_\infty^{-1}(\tau)} \right) \mid (\cZ_\infty^s, T^s_\infty)_{s \in J_+} \right] = \E \left[ f\left( R_\infty\mid_{T_\infty^{-1}(\tau)} \right) \mid (\Xi^s(\tau))_{s \in J}\right]. \]
Let us now evaluate the right-hand side of the latter identity. Since $(R_\infty) \ci (\cZ_\infty^s, T^s_\infty)_{s \in J}$ by induction hypothesis and because of~\eqref{eq:ddf}, we get that
$ R_\infty\mid_{T_\infty^{-1}(\tau)}$ only depends on $(\Xi^s(\tau), s \in J)$ through $T_\infty^{-1}(\tau)$, i.e.,
\[ \E \left[ f\left( R_\infty\mid_{T_\infty^{-1}(\tau)} \right) \mid (\Xi^s(\tau))_{s \in J} \right] = \E \left[ f\left( R_\infty\mid_{T_\infty^{-1}(\tau)} \right) \mid T_\infty^{-1}(\tau) \right]. \]
This concludes the proof of this step.
\\

\noindent \textit{Step $3$.} We now conclude the proof: Step 1 and~\eqref{eq:step-3} entail, using total expectation,
\begin{align*}
	\E \left[ f \left( R_\infty \right) g \left( (\cZ^s_\infty, T^s_\infty)_{s \in J_+} \right) \right] & = \E \left[ f \left( \left[ R_\infty\mid_{T_\infty^{-1}(\tau)}, \Gamma\right] \right) g \left( (\cZ^s_\infty, T^s_\infty)_{s \in J_+} \right) \right]\\
	& = \E \left[ g \left( (\cZ^s_\infty, T^s_\infty)_{s \in J_+} \right) \E \left[ f \left( \left[ R_\infty\mid_{T_\infty^{-1}(\tau)}, \Gamma\right] \right) \mid (\cZ^s_\infty, T^s_\infty)_{s \in J_+}\right] \right]\\
	& = \E \left[g \left( (\cZ^s_\infty, T^s_\infty)_{s \in J_+}\right) \E \left[ f \left( \left[ R_\infty\mid_{T_\infty^{-1}(\tau)}, \Gamma\right] \right) \mid T_\infty^{-1}(\tau) \right] \right].
\end{align*} 
At this point we want to apply Lemma~\ref{lem:sub}. We first note that 
$\left(R_\infty\mid_{(T_\infty^s)^{-1}(\tau)}, \Gamma\right)$ is equal in distribution to
$\left(R_\infty\mid_{(T_\infty^s)^{-1}(\tau)}, \tilde \Gamma\right)$ where $\tilde \Gamma$ is independent of $(R_\infty, T_\infty)$
according to Step 1. Further, since $R_\infty \ci T_\infty$, it follows that
$R_\infty$, $\tilde \Gamma$ and $T_\infty$ are mutually independent. Finally, 
applying Lemma~\ref{lem:sub} with $S_1 = R_\infty$, $S_2 = \tilde \Gamma$ and $\kappa = T_\infty^{-1}(\tau)$ we get that $[R_\infty\mid_{T_\infty^{-1}(\tau)}, \tilde \Gamma]$ is equal in distribution to $R_\infty$ and is independent from $T_\infty^{-1}(\tau)$, i.e.,
\[ \E \left[ \E \left[ f \left( \left[ R_\infty\mid_{T_\infty^{-1}(\tau)}, \tilde \Gamma\right] \right) \mid T_\infty^{-1}(\tau) \right] \right] \ = \ \E \left[ f \left( R_\infty \right) \right]. \]
This concludes the proof of Proposition~\ref{prop:more-ind}.

\begin{figure}[p!]
	\centering
		\begin{tikzpicture}
			[color=black]
				% omega_0
				\begin{scope}[shift={(0,0)}]
					\draw (0,0) -- (0,2);
					\stub{1.5}
					\stub{.5}
					\node[anchor=north] at (0,0) {$\omega_0$};
				\end{scope}
				% omega_1
				\begin{scope}[shift={(1,0)}]
					\draw (0,0) -- (0,1.5);
					\stub{1.2}
					\stub{.5}
					\node[anchor=north] at (0,0) {$\omega_1$};
				\end{scope}
				% omega_2
				\begin{scope}[shift={(2,0)}]
					\draw (0,0) -- (0,1.5);
					\stub{.9}
					\node[anchor=north] at (0,0) {$\omega_2$};
				\end{scope}
				% omega_3
				\begin{scope}[shift={(3,0)}]
					\draw (0,0) -- (0,1);
					\node[anchor=north] at (0,0) {$\omega_3$};
				\end{scope}
				% omega_4
				\begin{scope}[shift={(4,0)}]
					\draw (0,0) -- (0,2);
					\node[anchor=north] at (0,0) {$\omega_4$};
				\end{scope}
				% omega_5
				\begin{scope}[shift={(5,0)}]
					\draw (0,0) -- (0,4);
					\stub{3.5}
					\stub{2.5}
					\stub{1}
					\node[anchor=north] at (0,0) {$\omega_5$};
				\end{scope}
				% omega_6
				\begin{scope}[shift={(6,0)}]
					\draw (0,0) -- (0,2);
					\node[anchor=north] at (0,0) {$\omega_6$};
				\end{scope}
				% omega_7
				\begin{scope}[shift={(7,0)}]
					\draw (0,0) -- (0,1);
					\node[anchor=north] at (0,0) {$\omega_7$};
				\end{scope}
				% omega_8
				\begin{scope}[shift={(8,0)}]
					\draw (0,0) -- (0,1);
					\stub{1}
					\node[anchor=north] at (0,0) {$\omega_8$};
				\end{scope}
				% omega_9
				\begin{scope}[shift={(9,0)}]
					\draw (0,0) -- (0,1);
					\node[anchor=north] at (0,0) {$\omega_9$};
				\end{scope}
		\end{tikzpicture}
	\caption[Initial sequence of sticks]{Sequence of sticks used in the next figures: this sequence corresponds to one chronological tree.}
	\label{fig:sequential-construction-sticks}
\end{figure}
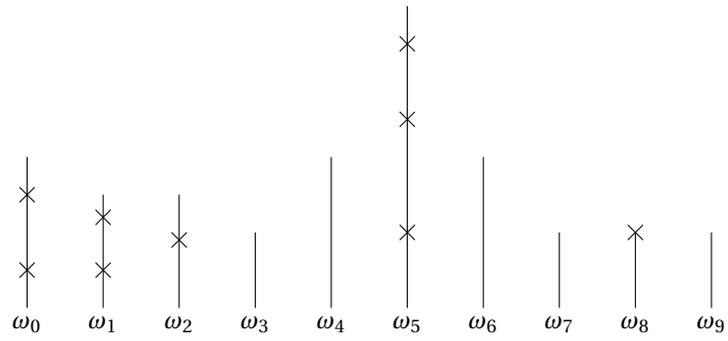

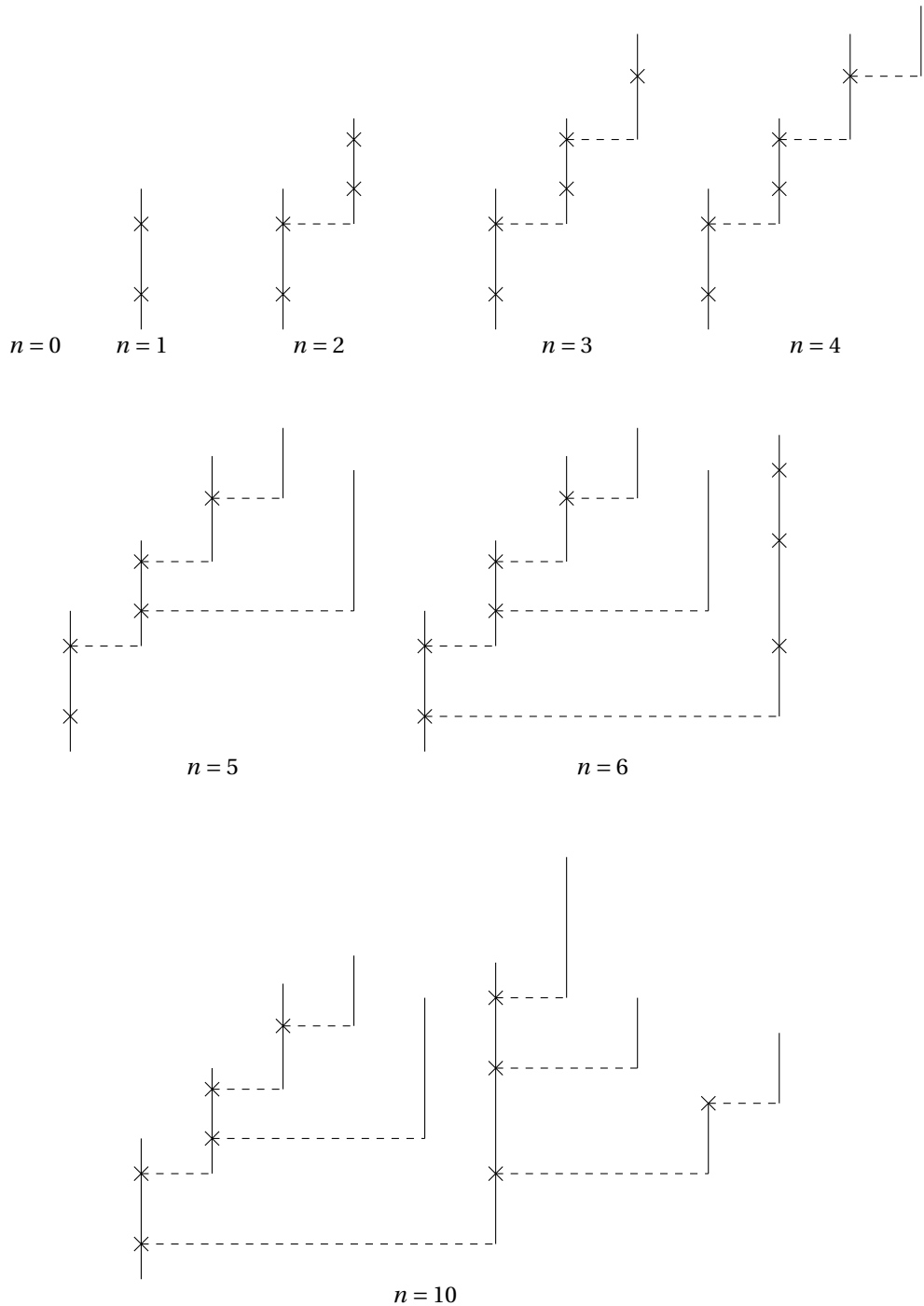
\begin{figure}[p!]
	\centering
		\begin{tikzpicture}
			[color=black]
			\begin{scope}[shift={(-.5,0)}]
				\node[anchor=north] at (0,0) {$n = 0$};
			\end{scope}
			% \n/\x/\shift = 1/0/1
			\begin{scope}[shift={(0,0)}]
				% root
				\begin{scope}[shift={(1,0)}]
					\draw (0,0) -- (0,2);
					\stub{1.5}
					\stub{.5}
				\end{scope}
				\node[anchor=north] at (1,0) {$n = 1$};
			\end{scope}
			% \n/\x/\shift = 2/.5/3
			\begin{scope}[shift={(3,0)}]
				% root
				\draw (0,0) -- (0,2);
				\stub{1.5}
				\draw[dashed] (0,1.5) -- (1,1.5);
				\begin{scope}[shift={(1,1.5)}]
					\draw (0,0) -- (0,1.5);
					\stub{1.2}
					\stub{.5}
				\end{scope}
				\stub{.5}
				\node[anchor=north] at (.5,0) {$n = 2$};
			\end{scope}
			% \n/\x/\shift = 3/1/6
			\begin{scope}[shift={(6,0)}]
				% root
				\draw (0,0) -- (0,2);
				\stub{1.5}
				% 1
				\draw[dashed] (0,1.5) -- (1,1.5);
				\begin{scope}[shift={(1,1.5)}]
					\draw (0,0) -- (0,1.5);
					\stub{1.2}
					% 2
					\draw[dashed] (0,1.2) -- (1,1.2);
					\begin{scope}[shift={(1,1.2)}]
						% 3
						\draw (0,0) -- (0,1.5);
						\stub{.9}
					\end{scope}
					\stub{.5}
				\end{scope}
				\stub{.5}
				\node[anchor=north] at (1,0) {$n = 3$};
			\end{scope}
			% \n/\x/\shift = 4/1.5/9
			\begin{scope}[shift={(9,0)}]
				% root
				\draw (0,0) -- (0,2);
				\stub{1.5}
				% 1
				\draw[dashed] (0,1.5) -- (1,1.5);
				\begin{scope}[shift={(1,1.5)}]
					\draw (0,0) -- (0,1.5);
					\stub{1.2}
					% 2
					\draw[dashed] (0,1.2) -- (1,1.2);
					\begin{scope}[shift={(1,1.2)}]
						% 3
						\draw (0,0) -- (0,1.5);
						\stub{.9}
						% 4
						\draw[dashed] (0,.9) -- (1,.9);
						\begin{scope}[shift={(1,.9)}]
							\draw (0,0) -- (0,1);
						\end{scope}
					\end{scope}
					\stub{.5}
				\end{scope}
				\stub{.5}

				\node[anchor=north] at (1.5,0) {$n = 4$};
			\end{scope}

			% \n/\x/\shift = 5/2/0
			\begin{scope}[shift={(0,-6)}]
				% root
				\draw (0,0) -- (0,2);
				\stub{1.5}
				% 1
				\draw[dashed] (0,1.5) -- (1,1.5);
				\begin{scope}[shift={(1,1.5)}]
					\draw (0,0) -- (0,1.5);
					\stub{1.2}
					% 2
					\draw[dashed] (0,1.2) -- (1,1.2);
					\begin{scope}[shift={(1,1.2)}]
						% 3
						\draw (0,0) -- (0,1.5);
						\stub{.9}
						% 4
						\draw[dashed] (0,.9) -- (1,.9);
						\begin{scope}[shift={(1,.9)}]
							\draw (0,0) -- (0,1);
						\end{scope}
					\end{scope}
					\stub{.5}
					% 5
					\draw[dashed] (0,.5) -- (3,.5);
					\begin{scope}[shift={(3,.5)}]
						\draw (0,0) -- (0,2);
					\end{scope}
				\end{scope}
				\stub{.5}
				\node[anchor=north] at (2,0) {$n = 5$};
			\end{scope}
			% \n/\x/\shift = 6/2.5/5
			\begin{scope}[shift={(5,-6)}]
				% root
				\draw (0,0) -- (0,2);
				\stub{1.5}
				% 1
				\draw[dashed] (0,1.5) -- (1,1.5);
				\begin{scope}[shift={(1,1.5)}]
					\draw (0,0) -- (0,1.5);
					\stub{1.2}
					% 2
					\draw[dashed] (0,1.2) -- (1,1.2);
					\begin{scope}[shift={(1,1.2)}]
						% 3
						\draw (0,0) -- (0,1.5);
						\stub{.9}
						% 4
						\draw[dashed] (0,.9) -- (1,.9);
						\begin{scope}[shift={(1,.9)}]
							\draw (0,0) -- (0,1);
						\end{scope}
					\end{scope}
					\stub{.5}
					% 5
					\draw[dashed] (0,.5) -- (3,.5);
					\begin{scope}[shift={(3,.5)}]
						\draw (0,0) -- (0,2);
					\end{scope}
				\end{scope}
				\stub{.5}
				% 6
				\draw[dashed] (0,.5) -- (5,.5);
				\begin{scope}[shift={(5,.5)}]
					\draw (0,0) -- (0,4);
					\stub{3.5}
					\stub{2.5}
					\stub{1}
				\end{scope}

				\node[anchor=north] at (2.5,0) {$n = 6$};
			\end{scope}

			% \n/\x/\shift = 10/4/1
				\begin{scope}[shift={(1,-13.5)}]
					% root
					\draw (0,0) -- (0,2);
					\stub{1.5}
					% 1
					\draw[dashed] (0,1.5) -- (1,1.5);
					\begin{scope}[shift={(1,1.5)}]
						\draw (0,0) -- (0,1.5);
						\stub{1.2}
						% 2
						\draw[dashed] (0,1.2) -- (1,1.2);
						\begin{scope}[shift={(1,1.2)}]
							% 3
							\draw (0,0) -- (0,1.5);
							\stub{.9}
							% 4
							\draw[dashed] (0,.9) -- (1,.9);
							\begin{scope}[shift={(1,.9)}]
								\draw (0,0) -- (0,1);
							\end{scope}
						\end{scope}
						\stub{.5}
						% 5
						\draw[dashed] (0,.5) -- (3,.5);
						\begin{scope}[shift={(3,.5)}]
							% 6
							\draw (0,0) -- (0,2);
						\end{scope}
					\end{scope}
					\stub{.5}
					% 7
					\draw[dashed] (0,.5) -- (5,.5);
					\begin{scope}[shift={(5,.5)}]
						\draw (0,0) -- (0,4);
						\stub{3.5}
						% 8
						\draw[dashed] (0,3.5) -- (1,3.5);
						\begin{scope}[shift={(1,3.5)}]
							\draw (0,0) -- (0,2);
						\end{scope}
						\stub{2.5}
						% 9
						\draw[dashed] (0,2.5) -- (2,2.5);
						\begin{scope}[shift={(2,2.5)}]
							\draw (0,0) -- (0,1);
						\end{scope}
						\stub{1}
						% 10
						\draw[dashed] (0,1) -- (3,1);
						\begin{scope}[shift={(3,1)}]
							\draw (0,0) -- (0,1);
							\stub{1}
							% 11
							\draw[dashed] (0,1) -- (1,1);
							\begin{scope}[shift={(1,1)}]
								\draw (0,0) -- (0,1);
							\end{scope}
						\end{scope}
					\end{scope}
					\node[anchor=north] at (4,0) {$n = 10$};
				\end{scope}
		\end{tikzpicture}
		\caption[Sequential construction of a chronological tree]{Sequential construction of the chronological tree from the sequence of sticks of Figure~\ref{fig:sequential-construction-sticks}: as long as there is a stub available, we graft the next stick at the highest one. At $n=10$ the construction is complete (there is no more stub available) and the next stick will therefore start the next tree in the forest.}
	\label{fig:sequential-construction}
\end{figure}

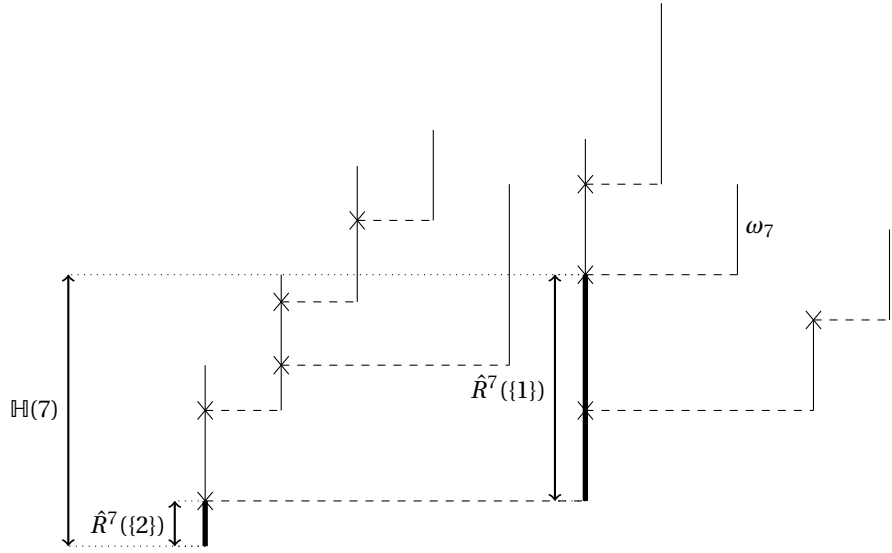
\begin{figure}[htbp]
	\centering
		\begin{tikzpicture}[yscale=1.2]
			\draw (0,0) coordinate (y0d) -- (0,2);
			\stub{1.5}
			% 1
			\draw[dashed] (0,1.5) -- (1,1.5);
			\begin{scope}[shift={(1,1.5)}]
				\draw (0,0) -- (0,1.5);
				\stub{1.2}
				% 2
				\draw[dashed] (0,1.2) -- (1,1.2);
				\begin{scope}[shift={(1,1.2)}]
					% 3
					\draw (0,0) -- (0,1.5);
					\stub{.9}
					% 4
					\draw[dashed] (0,.9) -- (1,.9);
					\begin{scope}[shift={(1,.9)}]
						\draw (0,0) -- (0,1);
					\end{scope}
				\end{scope}
				\stub{.5}
				% 5
				\draw[dashed] (0,.5) -- (3,.5);
				\begin{scope}[shift={(3,.5)}]
					% 6
					\draw (0,0) -- (0,2);
				\end{scope}
			\end{scope}
			\laststub{.5}
			% 7
			\draw[dashed] (0,.5) coordinate (y0u) -- (5,.5) coordinate (y1d);
			\begin{scope}[shift={(5,.5)}]
				\draw (0,0) -- (0,4);
				\stub{3.5}
				% 8
				\draw[dashed] (0,3.5) -- (1,3.5);
				\begin{scope}[shift={(1,3.5)}]
					\draw (0,0) -- (0,2);
				\end{scope}
				\laststub{2.5}
				% 9
				\draw[dashed] (0,2.5) coordinate (y1u) -- (2,2.5);
				\begin{scope}[shift={(2,2.5)}]
					\draw (0,0) -- (0,1) node [midway, right] {$\omega_7$};
				\end{scope}
				\stub{1}
				% 10
				\draw[dashed] (0,1) -- (3,1);
				\begin{scope}[shift={(3,1)}]
					\draw (0,0) -- (0,1);
					\stub{1}
					% 11
					\draw[dashed] (0,1) -- (1,1);
					\begin{scope}[shift={(1,1)}]
						\draw (0,0) -- (0,1);
					\end{scope}
				\end{scope}
			\end{scope}

			\draw [dotted] (y0d) -- ++ (-.4,0) coordinate (y0dp);
			\draw [dotted] (y0u) -- ++ (-.4,0) coordinate (y0up);
			\draw [<->, thick] (y0up) -- (y0dp) node [midway, left] {$\hat R^7(\{2\})$};

			\draw [dotted] (y1d) -- ++ (-.4,0) coordinate (y1dp);
			\draw [dotted] (y1u) -- ++ (-.4,0) coordinate (y1up);
			\draw [<->, thick] (y1up) -- (y1dp) node [midway, left] {$\hat R^7(\{1\})$};

			\draw [dotted] (y0d) -- ++ (-1.8,0) coordinate (y0dp);
			\draw [dotted] (y1u) -- (y0dp |- y1u) coordinate (y1up);
			\draw [<->, thick] (y1up) -- (y0dp) node [midway, left] {$\H(7)$};
		\end{tikzpicture}
	\caption[Spine decomposition]{Spine decomposition of the individual $7$: $\Pi(7) = \hat R^7(\{2\}) \epsilon_0 + \hat R^7(\{1\}) \epsilon_1$ where $\hat R^7(\{k\})$ is the age of the $k$th ancestor of $7$ when giving birth to the next individual on the spine. In particular, $\H(7) = \hat R^7(\{1\}) + \hat R^7(\{2\}) = \lvert \Pi(7) \rvert$ expressing that the chronological height of $7$ is obtained by summing up the chronological contribution of each ancestor on her spine.}
	\label{fig:y-spine}
\end{figure}

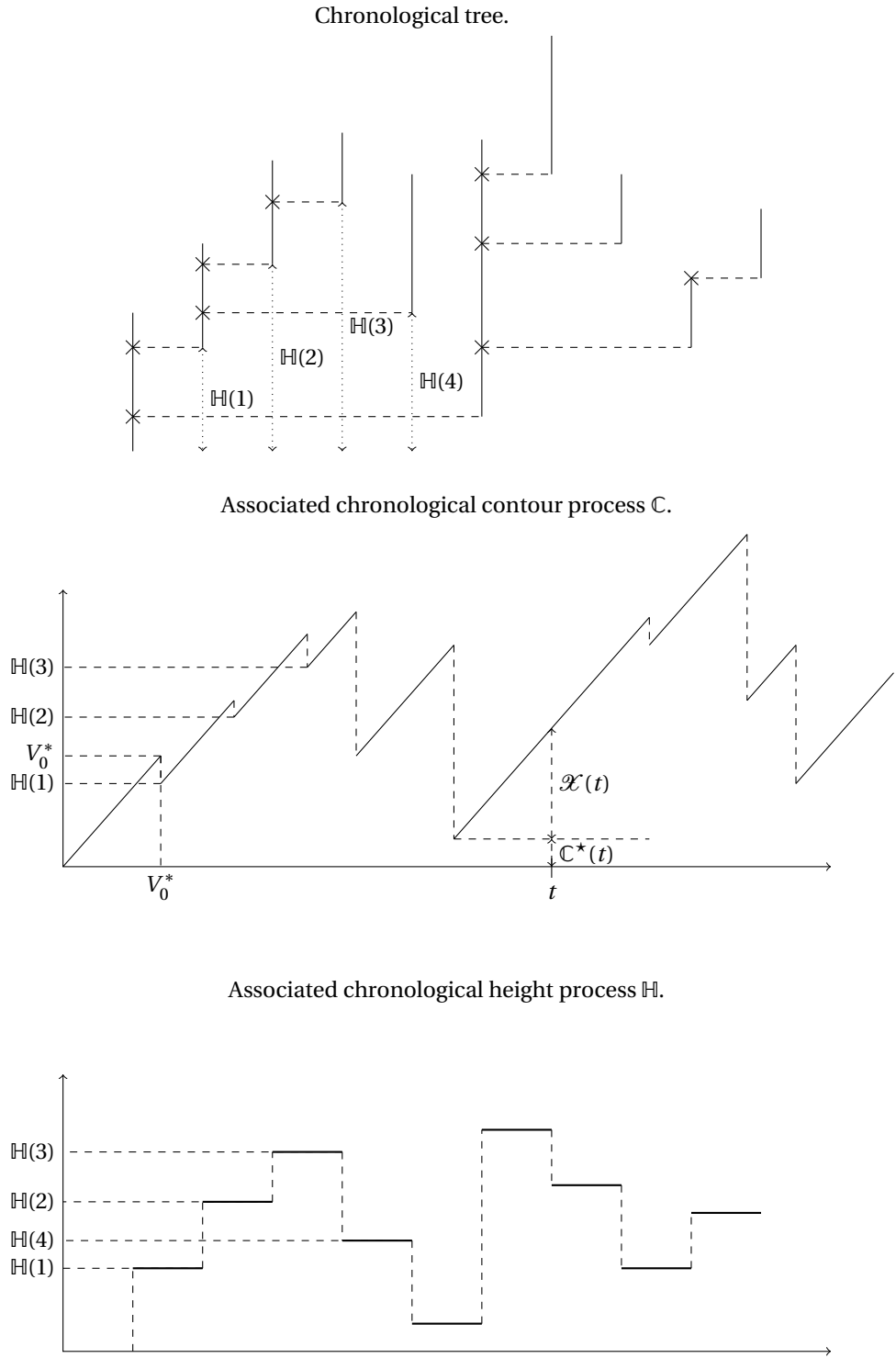
\begin{figure}[htbp]
	\centering
		\begin{tikzpicture}
			% root
			\coordinate (orig) at (0,0);
			\node[anchor=south] at (4,6) {Chronological tree.};
			\draw (0,0) -- (0,2);
			\stub{1.5}
			% 1
			\draw[dashed] (0,1.5) -- (1,1.5);
			\begin{scope}[shift={(1,1.5)}]
				\draw [dotted, <->] (0,0) -- (0,0 |- orig) node [midway, right] {$\H(1)$};
				\draw (0,0) -- (0,1.5);
				\stub{1.2}
				% 2
				\draw[dashed] (0,1.2) -- (1,1.2);
				\begin{scope}[shift={(1,1.2)}]
					\draw [dotted, <->] (0,0) -- (0,0 |- orig) node [midway, right] {$\H(2)$};
					% 3
					\draw (0,0) -- (0,1.5);
					\stub{.9}
					% 4
					\draw[dashed] (0,.9) -- (1,.9);
					\begin{scope}[shift={(1,.9)}]
						\draw [dotted, <->] (0,0) -- (0,0 |- orig) node [midway, right] {$\H(3)$};
						\draw (0,0) -- (0,1);
					\end{scope}
				\end{scope}
				\stub{.5}
				% 5
				\draw[dashed] (0,.5) -- (3,.5);
				\begin{scope}[shift={(3,.5)}]
					\draw [dotted, <->] (0,0) -- (0,0 |- orig) node [midway, right] {$\H(4)$};
					% 6
					\draw (0,0) -- (0,2);
				\end{scope}
			\end{scope}
			\stub{.5}
			% 7
			\draw[dashed] (0,.5) -- (5,.5);
			\begin{scope}[shift={(5,.5)}]
				\draw (0,0) -- (0,4);
				\stub{3.5}
				% 8
				\draw[dashed] (0,3.5) -- (1,3.5);
				\begin{scope}[shift={(1,3.5)}]
					\draw (0,0) -- (0,2);
				\end{scope}
				\stub{2.5}
				% 9
				\draw[dashed] (0,2.5) -- (2,2.5);
				\begin{scope}[shift={(2,2.5)}]
					\draw (0,0) -- (0,1);
				\end{scope}
				\stub{1}
				% 10
				\draw[dashed] (0,1) -- (3,1);
				\begin{scope}[shift={(3,1)}]
					\draw (0,0) -- (0,1);
					\stub{1}
					% 11
					\draw[dashed] (0,1) -- (1,1);
					\begin{scope}[shift={(1,1)}]
						\draw (0,0) -- (0,1);
					\end{scope}
				\end{scope}
			\end{scope}
			
			\begin{scope}[shift={(-1,-6)}, yscale=0.8]
				\node at (5.5, 6.5) {Associated chronological contour process $\C$.};
				\draw[<->] (11,0) -- (0,0) -- (0,5);
				\path (0,0) \foreach \u/\d/\n in
				{2/.5/0,
				1.5/.3/1,
				1.5/.6/2,
				1/2.6/3,
				2/3.5/4,
				4/.5/5,
				2/3/6,
				1/2.5/7,
				2/3.5/8}
				{coordinate (d\n) -- ++ (.7*\u,\u) coordinate (u\n) -- ++ (0,-\d)};
				\draw[name path=C] \foreach \m/\n in {0,1,2,...,8} {(d\m) -- (u\m)};
				\draw[dashed] \foreach \m/\n in
				{0/1,
				1/2,
				2/3,
				3/4,
				4/5,
				5/6,
				6/7,
				7/8}
				{(u\m) -- (d\n)};
				\foreach \n in {1,2,3} {\draw[dashed] (d\n) -- (0,0 |- d\n) node [left] {$\H(\n)$};}
				\draw[dashed] (u0) -- (0,0 |- u0) node [left] {$V^*_0$};
				\draw[dashed] (u0) -- (u0 |- 0,0) node [below] {$V^*_0$};
				
				\coordinate (t) at ($(d5 |- 0,0)!0.5!(d6 |- 0,0)$);
				\coordinate (baseline) at (t |- d5);
				\draw [dashed] (d5) -- (d6 |- d5);
				\draw ($(t)+(0,.2)$) -- ++ (0,-.4) node [below] {$t$};
				\path [name path=tup] (t) -- ++ (0,5);
				\path [name intersections={of=C and tup,by=mtmp}];
				\draw [dashed, <->] (t) -- (baseline) node [midway, right] {$\C^\star(t)$};
				\draw [dashed, <->] (baseline) -- (mtmp) node [midway, right] {$\X(t)$};
			\end{scope}

			\begin{scope}[shift={(-1,-13)}, yscale=0.8]
				\node at (5.5, 6.5) {Associated chronological height process $\H$.};
				\draw[<->] (11,0) -- (0,0) -- (0,5);
				\path (0,0) coordinate (v0) \foreach \h/\n in
				{1.5/1,
				1.2/2,
				.9/3,
				-1.6/4,
				-1.5/5,
				3.5/6,
				-1/7,
				-1.5/8,
				1/9,
				-2.5/10}
				{ -- ++ (1, 0) coordinate (u\n) -- ++ (0,\h) coordinate (v\n)};
				\foreach \n in {1,2,...,9} {\draw[dashed] (u\n) -- (v\n);}
				\foreach \n/\m in {1/2,2/3,3/4,4/5,5/6,6/7,7/8,8/9} {\draw[thick] (v\n) -- (u\m);}
				\draw[thick] (v9) -- ++ (1,0);
				\foreach \m/\n in {2/1,3/2,4/3,5/4}{\draw [dashed] (u\m) -- (0,0 |- u\m) node [left] {$\H(\n)$};}
			\end{scope}
		\end{tikzpicture}
	\caption[Chronological tree and associated processes]{Chronological height and contour processes associated to the chronological tree constructed from the sequence of sticks of Figure~\ref{fig:sequential-construction-sticks}.}
	\label{fig:chronological-processes}
\end{figure}

\begin{figure}[htbp]
	\centering
		\begin{tikzpicture}
			
			% root
			\node[anchor=south] at (0,2.5) {\begin{minipage}{0.8\textwidth}Genealogical tree associated with the chronological tree of Figure~\ref{fig:chronological-processes}.\end{minipage}};
			\begin{scope}[shift={(0,0)}]
				\tikzset{grow=up, level distance=6mm, level 3/.style={sibling distance=7mm}, level 2/.style={sibling distance=15mm}, level 1/.style={sibling distance=20mm}}
				\coordinate [fill] circle (2pt)
					child { [fill] circle (2pt)
						child { [fill] circle (2pt)
							child{ [fill] circle (2pt)
								child{ [fill] circle (2pt)}
						}
							child{ [fill] circle (2pt)}
							child{ [fill] circle (2pt)}
						}
						child { [fill] circle (2pt)
							child { [fill] circle (2pt)}
							child { [fill] circle (2pt)
								child{ [fill] circle (2pt)}
							}
						}
					}
				;
			\end{scope}

			\begin{scope}[shift={(-5.5,-5)}, yscale=.8]
				\node at (5.5, 5) {Associated genealogical height process $\cH$.};
				\draw[<->] (11,0) -- (0,0) -- (0,5);
				\path (0,0) coordinate (v0) \foreach \h/\n in
				{1/1,
				1/2,
				1/3,
				1/4,
				-2/5,
				-1/6,
				1/7,
				1/8,
				-1/9,
				1/10,
				-1/11,
				1/12,
				1/13}
				{ -- ++ (1, 0) coordinate (u\n) -- ++ (0,\h) coordinate (v\n)};
				\foreach \n in {1,2,...,9} {\draw[dashed] (u\n) -- (v\n);}
				\foreach \n/\m in {1/2,2/3,3/4,4/5,5/6,6/7,7/8,8/9} {\draw[thick] (v\n) -- (u\m);}
				\draw[thick] (v9) -- ++ (1,0);
			\end{scope}
			
			\begin{scope}[shift={(-5.5,-10)}, yscale=.8]
				\node at (5.5, 3.5) {Associated Lukasiewicz path $S$.};
				\draw[<->] (11,0) -- (0,0) -- (0,4);
				\path (0,0) coordinate (v0) \foreach \h/\n in
				{1/1,
				1/2,
				0/3,
				-1/4,
				-1/5,
				2/6,
				-1/7,
				-1/8,
				0/9,
				-1/10}
				{ -- ++ (1, 0) coordinate (u\n) -- ++ (0,\h) coordinate (v\n)};
				\foreach \n in {1,2,...,10} {\draw[dashed] (u\n) -- (v\n);}
				\foreach \n/\m in {1/2,2/3,3/4,4/5,5/6,6/7,7/8,8/9,9/10} {\draw[thick] (v\n) -- (u\m);}
				\draw[thick] (v10) -- ++ (1,0);
			\end{scope}
		\end{tikzpicture}
	\caption[Genealogical tree and associated processes]{The genalogical tree of the chronological tree of Figure~\ref{fig:chronological-processes}, together with the genealogical processes $S$ and $\cH$. The genealogical tree is obtained by resizing all the sticks to unit size and putting all the atoms at one. The genealogical height process is then obtained as before, but from the genealogical tree.}
	\label{fig:genealogical-processes}
\end{figure}

\begin{landscape}
	\begin{figure}[htbp]
		\centering
			\begin{tikzpicture}
				% axes
				\draw [->] (0,-2) -- (0,7.5);
				\draw [->] (-.2,0) -- (18,0);
				% S
				\path (0,0)
				\foreach \n/\d/\j in {
				1/1/3.3,
				2/2/1,
				3/1/6,
				4/4/1.5,
				5/2/1.5,
				6/2/4,
				7/2/1.5,
				8/1/4.5,
				9/3/1}
				{coordinate (l\n) -- ++ (\d, -\d) coordinate (r\n) -- ++ (0,\j)};
				\draw [name path=S] \foreach \n in {1,...,9} {(l\n) -- (r\n)};
				\draw [dashed] \foreach \n/\o in {1/2,2/3,3/4,4/5,5/6,6/7,7/8,8/9} {(r\n) -- (l\o)};
				\coordinate (m) at ($(l5)!0.5!(r5)$);
				\draw[dashed] (m) -- (m |- 0,0) node [below] {$m$};
				\draw [dashed] (0,0 |- m) node [left] (S) {$S(m)$} -- ++ (16,0);
				\draw [dashed] (0,0 |- l4) node [left] (sup) {$\sup_{[0,m]} S$} -- ++ (16,0);
				\draw [dashed] (r8) -- (r8 |- 0,0) node [below] {$T(T^{-1}(m))$};
				\draw [<->, dashed] ($(r8 |- S)+(.7,0)$) -- ($(r8 |- sup)+(.7,0)$) node [midway, right] {$X(m)$};
				\draw [dashed] ($(l7 |- l9)+(-1,0)$) -- (l9);
				\draw [<->, dashed] ($(l9)-(.5,0)$) -- ($(l9 |- sup)-(.5,0)$) node [midway, left] {$\Delta \cZ (T^{-1}(m))$};
				\draw [dashed] ($(l7)+(-1,0)$) -- (l9 |- l7);
				\draw [<->, dashed] ($(l7 |- l9)+(-1,0)$) -- ($(l7)+(-1,0)$) node [midway, left] {$\Delta \cZ^m \circ (\cZ^m)^{-1} (X(m))$};
			\end{tikzpicture}
		\caption[Illustration of two key relations]{Proof of the relation $\Theta_{T^{-1}(m)}(W) = \Theta_{(\cZ^m)^{-1} \circ X(m)}(W^m)$ for Lemma~\ref{prop:discrete-snake}. By definition, $\Theta_{T^{-1}(m)}(W)$ is computed from $S$ shifted at time $T(T^{-1}(m))$ which is the first ascending ladder height time after $m$. In terms of the process shifted at time $m$, this means that we have to wait for $T(T^{-1}(m)) - m$ which corresponds to skipping $(\cZ^m)^{-1} \circ X(m)$ excursions. On the picture, we have for instance $(\cZ^m)^{-1} \circ X(m) = 4$. Indeed, $\cZ^{-1}(\ell)$ is by definition the number of excursions needed to reach level $\ell$, and $X(m)$ is precisely the depth of the ``valley'' in which $m$ sits and from which the process shifted at time $m$ needs to exit in order to coincide with the process shifted at time $T(T^{-1}(m))$.}
		\label{fig:discrete-snake}
	\end{figure}
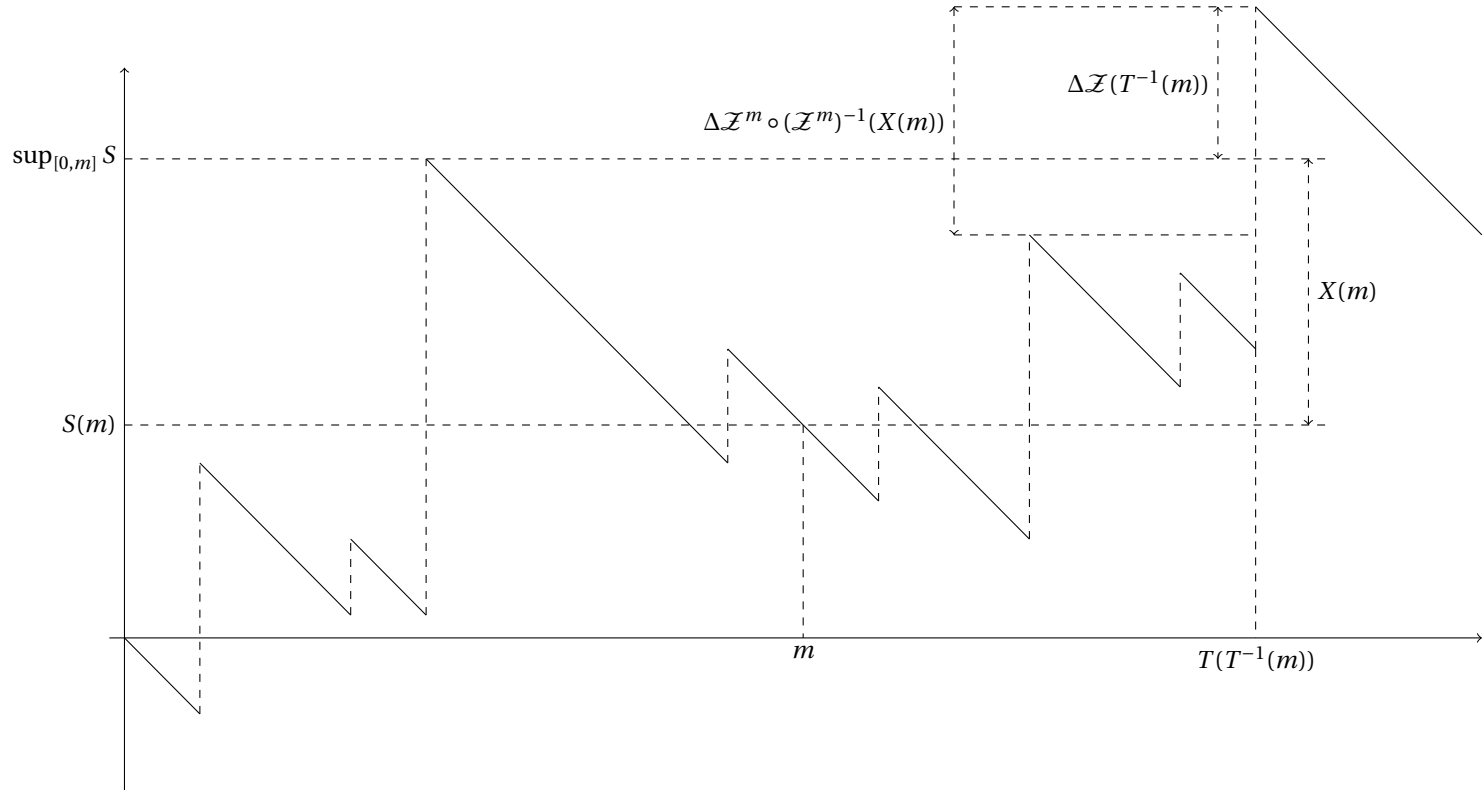
\end{landscape}


\begin{thebibliography}{10}

\bibitem{Abraham02:0}
Romain Abraham and Laurent Serlet.
\newblock Poisson snake and fragmentation.
\newblock {\em Electron. J. Probab.}, 7:no. 17, 15 pp. (electronic), 2002.

\bibitem{Aldous93:0}
David Aldous.
\newblock The continuum random tree. {III}.
\newblock {\em Ann. Probab.}, 21(1):248--289, 1993.

\bibitem{Bertoin96:0}
Jean Bertoin.
\newblock {\em L\'evy processes}, volume 121 of {\em Cambridge Tracts in
  Mathematics}.
\newblock Cambridge University Press, Cambridge, 1996.

\bibitem{Bingham89:0}
N.~H. Bingham, C.~M. Goldie, and J.~L. Teugels.
\newblock {\em Regular variation}, volume~27 of {\em Encyclopedia of
  Mathematics and its Applications}.
\newblock Cambridge University Press, Cambridge, 1989.

\bibitem{Cloez:0}
Bertrand Cloez and Beno\^it Henry.
\newblock {Markovian tricks for non-Markovian trees: contour process,
  extinction and scaling limits}.
\newblock arXiv 1801.03284.

\bibitem{Davila-Felipe15:0}
Miraine D\'avila~Felipe and Amaury Lambert.
\newblock Time reversal dualities for some random forests.
\newblock {\em ALEA Lat. Am. J. Probab. Math. Stat.}, 12(1):399--426, 2015.

\bibitem{Delaporte15:0}
C\'ecile Delaporte.
\newblock L\'evy processes with marked jumps {I}: {L}imit theorems.
\newblock {\em J. Theoret. Probab.}, 28(4):1468--1499, 2015.

\bibitem{Duquesne:0}
Thomas Duquesne.
\newblock The coding of compact real trees by real valued functions.
\newblock arXiv 0604106.

\bibitem{Duquesne02:0}
Thomas Duquesne and Jean-Fran{\c{c}}ois Le~Gall.
\newblock Random trees, {L}\'evy processes and spatial branching processes.
\newblock {\em Ast\'erisque}, (281):vi+147, 2002.

\bibitem{Evans08:0}
Steven~N. Evans.
\newblock {\em Probability and real trees}, volume 1920 of {\em Lecture Notes
  in Mathematics}.
\newblock Springer, Berlin, 2008.
\newblock Lectures from the 35th Summer School on Probability Theory held in
  Saint-Flour, July 6--23, 2005.

\bibitem{Gittenberger98:0}
Bernhard Gittenberger.
\newblock Convergence of branching processes to the local time of a {B}essel
  process.
\newblock In {\em Proceedings of the {E}ighth {I}nternational {C}onference
  ``{R}andom {S}tructures and {A}lgorithms'' ({P}oznan, 1997)}, volume~13,
  pages 423--438, 1998.

\bibitem{Gittenberger03:0}
Bernhard Gittenberger.
\newblock A note on: ``{S}tate spaces of the snake and its tour---convergence
  of the discrete snake'' [{J}. {T}heoret. {P}robab. {\bf 16} (2003), no. 4,
  1015--1046; mr2033196] by {J}.-{F}.\ {M}arckert and {A}.\ {M}okkadem.
\newblock {\em J. Theoret. Probab.}, 16(4):1063--1067 (2004), 2003.

\bibitem{Jacod03:0}
Jean Jacod and Albert~N. Shiryaev.
\newblock {\em Limit theorems for stochastic processes}, volume 288 of {\em
  Grundlehren der Mathematischen Wissenschaften [Fundamental Principles of
  Mathematical Sciences]}.
\newblock Springer-Verlag, Berlin, second edition, 2003.

\bibitem{Janson05:0}
Svante Janson and Jean-Fran{\c{c}}ois Marckert.
\newblock Convergence of discrete snakes.
\newblock {\em J. Theoret. Probab.}, 18(3):615--647, 2005.

\bibitem{Kurtz91:1}
Thomas~G. Kurtz.
\newblock Random time changes and convergence in distribution under the
  {M}eyer-{Z}heng conditions.
\newblock {\em Ann. Probab.}, 19(3):1010--1034, 1991.

\bibitem{Lambert10:0}
Amaury Lambert.
\newblock The contour of splitting trees is a {L}\'evy process.
\newblock {\em Ann. Probab.}, 38(1):348--395, 2010.

\bibitem{Lambert15:0}
Amaury Lambert and Florian Simatos.
\newblock Asymptotic {B}ehavior of {L}ocal {T}imes of {C}ompound {P}oisson
  {P}rocesses with {D}rift in the {I}nfinite {V}ariance {C}ase.
\newblock {\em J. Theoret. Probab.}, 28(1):41--91, 2015.

\bibitem{Lambert13:0}
Amaury Lambert, Florian Simatos, and Bert Zwart.
\newblock Scaling limits via excursion theory: {I}nterplay between
  {Crump-Mode-Jagers branching processes and Processor-Sharing queues}.
\newblock {\em Ann. Appl. Probab.}, 23(6):2357--2381, 2013.

\bibitem{Le-Gall98:0}
Jean-Fran{\c{c}}ois Le~Gall and Yves Le~Jan.
\newblock Branching processes in {L}\'evy processes: the exploration process.
\newblock {\em Ann. Probab.}, 26(1):213--252, 1998.

\bibitem{Marckert03:0}
Jean-Fran\c{c}ois Marckert and Abdelkader Mokkadem.
\newblock The depth first processes of {G}alton-{W}atson trees converge to the
  same {B}rownian excursion.
\newblock {\em Ann. Probab.}, 31(3):1655--1678, 2003.

\bibitem{Marckert03:1}
Jean-Fran\c{c}ois Marckert and Abdelkader Mokkadem.
\newblock States spaces of the snake and its tour---convergence of the discrete
  snake.
\newblock {\em J. Theoret. Probab.}, 16(4):1015--1046 (2004), 2003.

\bibitem{Richard13:0}
Mathieu Richard.
\newblock L\'evy processes conditioned on having a large height process.
\newblock {\em Ann. Inst. Henri Poincar\'e Probab. Stat.}, 49(4):982--1013,
  2013.

\bibitem{Richard14:0}
Mathieu Richard.
\newblock Splitting trees with neutral mutations at birth.
\newblock {\em Stochastic Process. Appl.}, 124(10):3206--3230, 2014.

\bibitem{Schertzer:0}
Emmanuel Schertzer and Florian Simatos.
\newblock {Height and contour processes of Crump-Mode-Jagers forests (I):
  general distribution and scaling limits in the case of short edges}.
\newblock arXiv 1506.03192.

\bibitem{Vatutin79:1}
V.~A. Vatutin.
\newblock A new limit theorem for a critical {B}ellman-{H}arris branching
  process.
\newblock {\em Mat. Sb. (N.S.)}, 109(151)(3):440--452, 480, 1979.

\bibitem{Whitt80:0}
Ward Whitt.
\newblock Some useful functions for functional limit theorems.
\newblock {\em Math. Oper. Res.}, 5(1):67--85, 1980.

\end{thebibliography}
\end{document}